\documentclass[reqno]{amsart}
\usepackage{a4,amsmath,amsfonts,hyperref,cite}
\usepackage{geometry}
\usepackage{graphicx,epstopdf,color}
\usepackage{array}
\usepackage{booktabs}
\usepackage{subfigure,tikz}
\usepackage{rotating}

\usepackage{latexsym, amsmath, verbatim, graphics, epsfig, amssymb, color, mathrsfs}
\usepackage{setspace,float}
\usepackage{graphicx,subfigure}
\usepackage{caption,color}
\usepackage{tikz}
\usepackage{array}
\usepackage{hyperref}
\usepackage{multirow}
\usepackage{pdflscape}
\usepackage{afterpage}
\usepackage{capt-of}
\usepackage{lipsum}
\newtheorem{theorem}{Theorem}

\newtheorem{remark}{Remark}

\newtheorem{lemma}{Lemma}

\newtheorem{case}{Case}

\usepackage{setspace}

\allowdisplaybreaks

\begin{document}
	\begin{center}
		\textbf{\Large Bifurcation Patterns and Chaos Control in\\ Discrete-Time Coral Reef Model}\\
		\vskip 0.3cm
			{\large M. Priyanka$^a$,  P. Muthukumar$^{a}$\footnote{Corresponding Author;
					Email: pmuthukumargri@gmail.com, Phone: 91-451-2452371,
					Fax: 91-451-2454466.\\  priyankamurugangri@gmail.com (M. Priyanka).}, \\
				\vskip 0.1cm 
				$^a$Department of Mathematics,\\
				The Gandhigram Rural Institute (Deemed to be University),\\
				Gandhigram - 624 302, Tamil Nadu, India.}\\
				\vskip 1.3cm
		\end{center}

\begin{quote}
 \textbf{Abstract:}
 The reduction in coral reef densities, characterized by the proliferation of macroalgae, has emerged as a global threat. In this paper, we present a discrete-time coral reef dynamical model that incorporates macroalgae. We explore all ecologically possible equilibrium points for the proposed model. The conditions for the local stability of the interior equilibrium point are analyzed, which represents the coexistence of both coral and macroalgae. Furthermore, we investigate the model's behavior using the center manifold theorem and bifurcation theory. Our analysis reveals that the model undergoes codimension-one bifurcations, specifically period-doubling and Neimark-Sacker bifurcations. To address the chaos resulting from the emergence of the Neimark-Sacker bifurcation, we apply the OGY feedback control method and a hybrid control methodology. Finally, we provide numerical simulations not only to validate the obtained results but also to demonstrate the complex dynamic behaviors that arise. These behaviors include reversal period-doubling bifurcation, period-4, 8, and 24 bubble bifurcations, as well as chaotic behavior. 
\end{quote}


 
\section{Introduction}
The biodiversity of coral reefs plays a crucial role in maintaining the environment's health and holds significant economic value. Due to their ecosystem functions, which include coastal storm defense and the preservation of fisheries and marine biodiversity, coral reefs are vital to aquatic environments. Additionally, coral reef-dependent fishing and tourism provide livelihoods for millions of people \cite{paper6-Li}. However, reefs are experiencing a severe decline; many of these populations are already considered seriously damaged and remain at risk of disappearing soon \cite{paper6-Mora}. Pristine reefs no longer exist, and despite local efforts to safeguard coral reefs over the past 30 years, regional-scale declines persist.

Scientists believe that millions of unknown organisms thrive in and around reefs. Future discoveries of new treatments depend on this biodiversity \cite{paper6-Doering,paper6-Ranjit}. Coral reef animals and plants are now used to create medicines that potentially treat diseases such as cancer, arthritis, human bacterial infections, viruses, and other conditions. Many studies have concentrated on the effects of macroalgae on coral reefs and the potential role of these interactions leading to coral demise \cite{paper6-Clements}.

The interaction between these species plays a significant role in the oxygen levels in the atmosphere \cite{paper6-Jokiel}. Consequently, fluctuations in the population of these species can contribute to global warming, potentially resulting in devastating consequences for life on Earth, even risking and leading to the extinction of various animal species, including humans. Macroalgal forests are highly productive, providing food and habitat for numerous species, including fish \cite{paper6-Melis,paper6-Panja}. Hence, it is crucial to mathematically analyze the interaction between coral and macroalgae for the continued existence of these species. The ongoing study of coral reefs and macroalgae presents current and future threats to humans, along with novel research ideas to support the management of these essential natural resources.

There are two categories of mathematical models for population dynamics: those defined by continuous-time differential equations and those defined by discrete-time difference equations. Generations with overlapping growth processes, such as the human population, are described by nonlinear differential equations. In other cases, population growth occurs over discrete time intervals with non-overlapping generations, as seen in 13-year periodical cicadas. Nonlinear difference equations mathematically describe this situation \cite{paper6-May}. 

Numerous recent studies suggest that discrete-time equations describe marine
dynamics more accurately and naturally \cite{paper6-Zhang, paper6-Han}. Discrete-time models significantly reduce computing time and retain the essential characteristics of their related continuous-time counterparts \cite{paper6-Xu,paper6-Alamin}. Moreover, they exhibit richer dynamics than continuous-time models, leading to chaotic behaviors from a biological perspective \cite{paper6-Zhu, paper6-Zelinka}. Therefore, in recent years, many researchers have focused on discrete-time population models. However, there remains a research gap in the discrete-time analysis of coral reef models.

The primary purpose of modeling population dynamics is to identify the controlling parameters and predict the expected outcomes when environmental parameters change \cite{paper6-Murray}. Model parameters determine the dynamical properties of equilibriums in nonlinear dynamical systems. Changes in the qualitative structure of the model corresponding to parameter fluctuations are known as bifurcation points \cite{paper6-Luo1}. Therefore, it is crucial to identify the variables that can cause instability and unpredictable behavior, such as chaos.

Moreover, detecting codimension-one bifurcations in a model allows one to anticipate global phenomena, including hysteresis, invariant tori, limit cycles, homoclinic bifurcations, and chaotic attractors \cite{paper6-Kuznetsov}. For this reason, many authors concentrate on bifurcation analysis and chaos control in population dynamical models \cite{paper6-Mukherjee,paper6-Rahman}.

Several researchers have studied chaotic discrete dynamical systems. Researchers in \cite{paper6-Singh} employs three chaos control methods to examine the fractional-ordered discretized two-dimensional Leslie-Gower prey-predator model. In \cite{paper6-Akhtar}, a non-standard finite difference scheme discretizes a continuous-time Leslie prey-predator model. Additionally, authors in \cite{paper6-Alamin} comprehensively investigate a discrete food web model enriched with mate-finding Allee effects and hunting cooperation to explore resilience, chaos, and bifurcations.

In this paper, we contribute to understanding a discrete-time coral reef model. To date, there has been no comprehensive qualitative study of a discrete-time coral reef model that considers the significance of growth rate on macroalgae. Therefore, this paper focuses on the stability and bifurcation analysis of a two-dimensional discrete-time coral reef model. This study employs center manifold theorem \cite{paper6-Robinson,paper6-Layek,paper6-Carr} and bifurcation theory \cite{paper6-Guckenheimer,paper6-Robinson,paper6-Layek} to analyze the discrete-time coral reef model within the interior of $\mathbb{R}^2_+$. We demonstrate the existence of flip and Neimark-Sacker bifurcation in the proposed model. The primary contribution of this paper is outlined as follows:
\begin{itemize}
	\item The proposed coral reef model comprises macroalgae, coral reefs, and algal turfs. The conditions for the local stability of the model's steady states are derived.
	\item The existence of flip and Neimark-Sacker bifurcation are discussed with the help of the center manifold theorem and bifurcation theory.
	\item OGY feedback control method is implemented to control the chaos resulting from the emergence of Neimark-Sacker bifurcation. Then, the theoretical analysis is validated using numerical simulations.
	\item Similar to existing studies \cite{paper6-Li,paper6-Fattahpour,paper6-Blackwood,paper6-Blackwood1,paper6-Rivero}, we investigated the interaction between coral reefs and macroalgae with discrete-time. Observing the model reveals complex dynamic behaviors and demonstrates a delayed response in bifurcation when subjected to an incremental increase in the intrinsic growth rate of macroalgae.
\end{itemize}

Paper organization: Section \ref{paper6-sec1} includes model formation and description. Section \ref{paper6-sec2} studies the existence and stability of the model's steady states. In section \ref{paper6-sec3}, we demonstrate that the model undergoes flip and Neimark-Sacker bifurcation around the interior equilibrium point. In Section \ref{paper6-sec4}, the OGY method and hybrid control procedures are employed to control the chaos due to the appearance of the Neimark-Sacker bifurcation. Finally, results and discussions were given in Section \ref{paper6-sec5}.

\section{Formulation of model with key assumptions}\label{paper6-sec1}
In this section, we start with a basic conceptual framework of the coral reef model, including corals, macroalgae, and short algal turfs. 
\begin{align}\label{paper6-mod0}
	\frac{dM}{dt}&=\underbrace{rM\left(1-\frac{M}{k}\right)}_{\textnormal{Logistic growth}}+\underbrace{aMC}_{\substack{\textnormal{Coral overgrown} \\ \textnormal{by macroalgae}}}
	-\underbrace{\frac{gM}{M+S}}_{\textnormal{Grazing rate}}+\underbrace{\gamma MS}_{\substack{\textnormal{Macroalgae} \\ \textnormal{spread rate}}},\nonumber\\
	\frac{dC}{dt}&=\underbrace{\alpha S C}_{\substack{\textnormal{Coral overgrows} \\ \textnormal{algal turfs}}}-\underbrace{dC}_{\textnormal{Coral mortality}}-\underbrace{aMC}_{\substack{\textnormal{Coral overgrown} \\ \textnormal{by macroalgae}}},\nonumber\\
	\frac{dS}{dt}&=-\underbrace{rM\left(1-\frac{M}{k}\right)}_{\textnormal{Logistic growth}}+\underbrace{\frac{gM}{M+S}}_{\textnormal{Grazing rate}}-\underbrace{\gamma MS}_{\substack{\textnormal{Macroalgae} \\ \textnormal{spread rate}}}-\underbrace{\alpha S C}_{\substack{\textnormal{Coral overgrows} \\ \textnormal{algal turfs}}}+\underbrace{dC}_{\textnormal{Coral mortality}},
\end{align}
where $M$, $C$, and $S$ denote the densities of macroalgae, corals, and algal turfs, respectively. For constructing the mathematical model \eqref{paper6-mod0}, the following assumptions are made:
\begin{itemize}
	\item[(i)] It is assumed that a specified region of the seabed is fully sheltered by macroalgae $(M)$, coral $(C)$, and algal turfs $(S)$ so that $M + C + S$ is kept constant over time, by a rescaling, lets assume that $M + C + S = 1$. Hence, algal turfs is defined by $S=1-M-C$ and accordingly $\frac{dS}{dt}$ is given by $-\frac{dM}{dt}-\frac{dC}{dt}$.
	\item[(ii)] It is assumed that corals recruit and overgrow algal turfs at a rate $\alpha$, they have a natural mortality rate of $d$, are overgrown by macroalgae at a rate $a$, and macroalgae spread vegetatively over algal turfs at a rate $\gamma$. 
	\item[(iii)] Algal turfs are believed to be the space recolonized after coral mortality. Additionally, the ratio $\frac{M}{M+S}$ represents the percentage of grazing $g$ that affects macroalgae. 
	\item[(iv)] Here, we incorporate logistic macroalgae growth into the proposed coral reef model with an intrinsic growth rate $r$ and a time-varying carrying capacity $k$.
	\item[(v)] Contrasting growth dynamics between macroalgae (logistic growth) and algal turfs (negative logistic growth) reflect competitive interactions within the ecosystem.
\end{itemize}
Biological description of the proposed model \eqref{paper6-mod0} parameters are given in Table \ref{paper6-t1}. The pictorial representation of the coral reef model is shown in Figure \ref{paper6-coralfig}.
\begin{figure}[h!]
	\begin{center}
		\includegraphics[width=9.5cm,height=5cm]{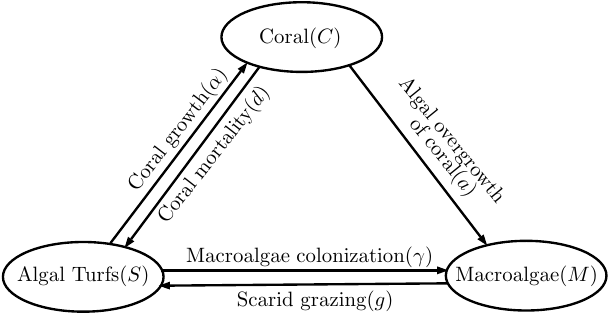}
		\caption{The schematic representation of coral reef model.}
		\label{paper6-coralfig}
	\end{center}
\end{figure}

\begin{table}[h!] 
	\begin{center}
		\caption{Model \eqref{paper6-mod0} parameter description.}
		\begin{tabular}{ |p{2.3cm}| p{7cm}| }
			\hline
			\textbf{Parameter} &  \textbf{Interpretation}\\
			\hline
			$r$ & intrinsic growth rate\\
			
			$k$ & carrying capacity of macroalgae\\
			
			$a$ & rate at which coral overgrows macroalgae\\
			
			$g$ & macroalgae grazing rate\\
			
			$\gamma$ & spread rate of macroalgae over algal turfs\\
			
			$\alpha$ & rate at which coral overgrows algal turfs\\
			
			$d$ &  natural mortality rate of coral\\
			\hline
		\end{tabular} \label{paper6-t1}
	\end{center}
\end{table}

Therefore, two equations are sufficient to describe the dynamics of this model (see \cite{paper6-Li} for more details). From those mentioned above and setting $S=1-M-C$, the coral dynamics are given by
\begin{align}\label{paper6-mod1}
	\frac{dM}{dt}&=rM\left(1-\frac{M}{k}\right)+aMC-\frac{gM}{1-C}+\gamma M(1-M-C),\nonumber\\
	\frac{dC}{dt}&=\alpha (1-M-C) C-dC-aMC.
\end{align}

To discretize the continuous-time coral reef model \eqref{paper6-mod1}, we use the forward Euler technique, for $n=0,1,2,\cdots ,$
\begin{align}\label{paper6-mod11}
	M_{n+1}&=M_n+\delta\left[rM_n\left(1-\frac{M_n}{k} \right)+aM_nC_n-\frac{gM_n}{1-C_n}+\gamma M_n (1-M_n-C_n) \right],\nonumber\\
	C_{n+1}&=C_n+\delta\Big[\alpha (1-M_n-C_n)C_n-dC_n-aM_nC_n \Big],
\end{align}
such that $M_0=M(0)\geq 0$, $C_0=C(0) \geq 0,$ and $\delta>0$ is the step size. We have discretized the model to emphasize the significance of step size in chaos control analysis for the dynamic behavior of coral reef population model. The corresponding map representation is then provided as:
\begin{align}\label{paper6-mod2}
	M &\longrightarrow M+\delta\left[rM\left(1-\frac{M}{k} \right)+aMC-\frac{gM}{1-C}+\gamma M(1-M-C) \right],\nonumber\\
	C & \longrightarrow C+\delta\Big[\alpha (1-M-C)C-dC-aMC \Big].
\end{align}

\section{Stability analysis of two-component coral reef model}\label{paper6-sec2}

This section discusses the existence of fixed points and their stability analysis for the proposed model \eqref{paper6-mod2}. The equilibrium points of the proposed coral reef model are listed below: 

\begin{itemize}
	\item[(i)] Trivial equilibrium $E_0(0,0).$
	\item[(ii)] Axial equilibrium $E_{A_1}(M_A, 0),$ where $M_A$=$\frac{(r-g+\gamma)k}{(r+k \gamma)}$.
	\item[(iii)] Axial equilibrium $E_{A_2}(0, C_A),$ where $C_A=1-\frac{d}{ \alpha}$.
	\item[(iv)] Interior equilibrium, $E^*(M^*, C^*)$ in the presence of both populations is the intersection point of two following nontrivial nullclines in the first quadrant of $\mathbb{R}^2$:
	\begin{align*}
		r\left(1-\frac{M}{k}\right)+aC-\frac{g}{1-C}+\gamma (1-M-C)&=0,\\
		\alpha (1-M-C)-d-aM&=0.
	\end{align*}
\end{itemize}
	\begin{figure}[h!]
		\begin{center}
		\includegraphics[width=10cm,height=5.5cm]{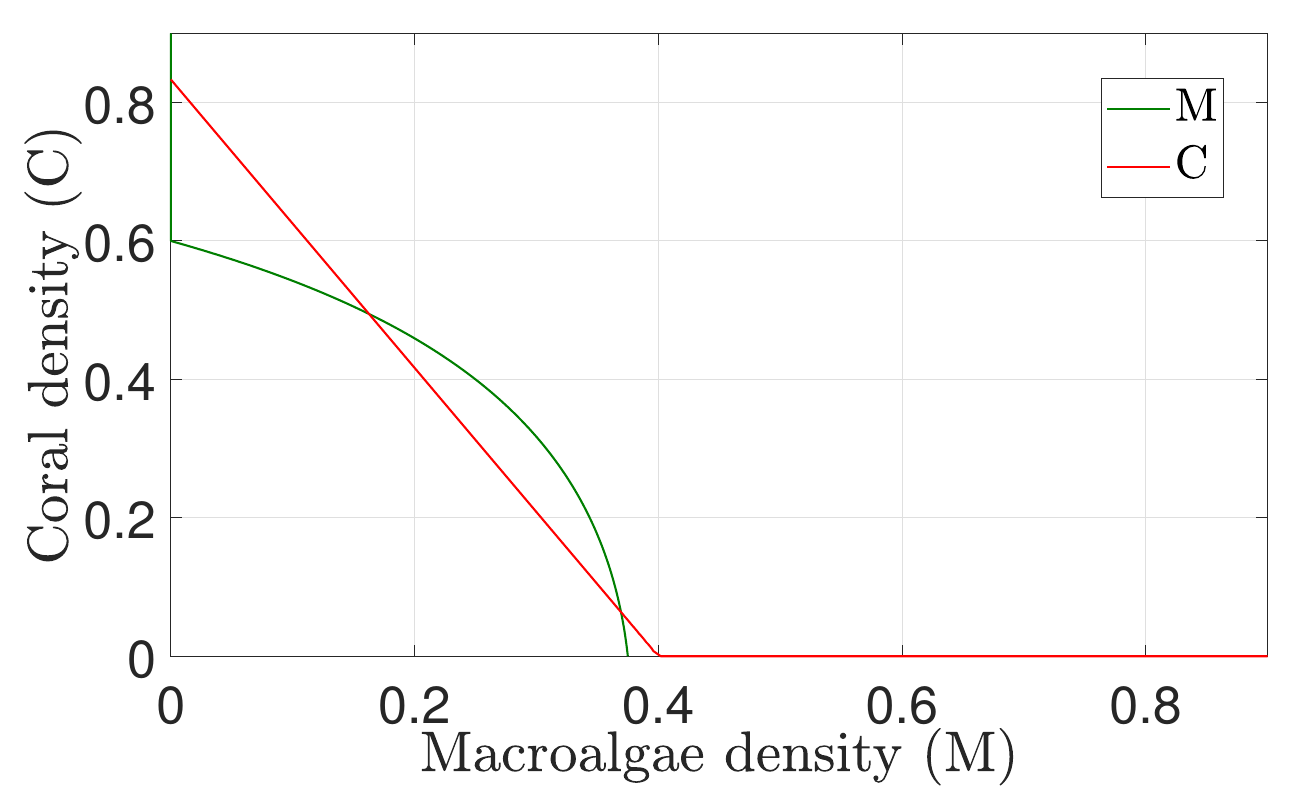}
		\caption{Nullclines indicates the occurrence of two equilibrium points of model \eqref{paper6-mod2}.}
		\label{paper6-nullcline-0.2}
	\end{center}
	\end{figure}
	\begin{figure}[h!]
		\begin{center}
		\includegraphics[width=10cm,height=5.5cm]{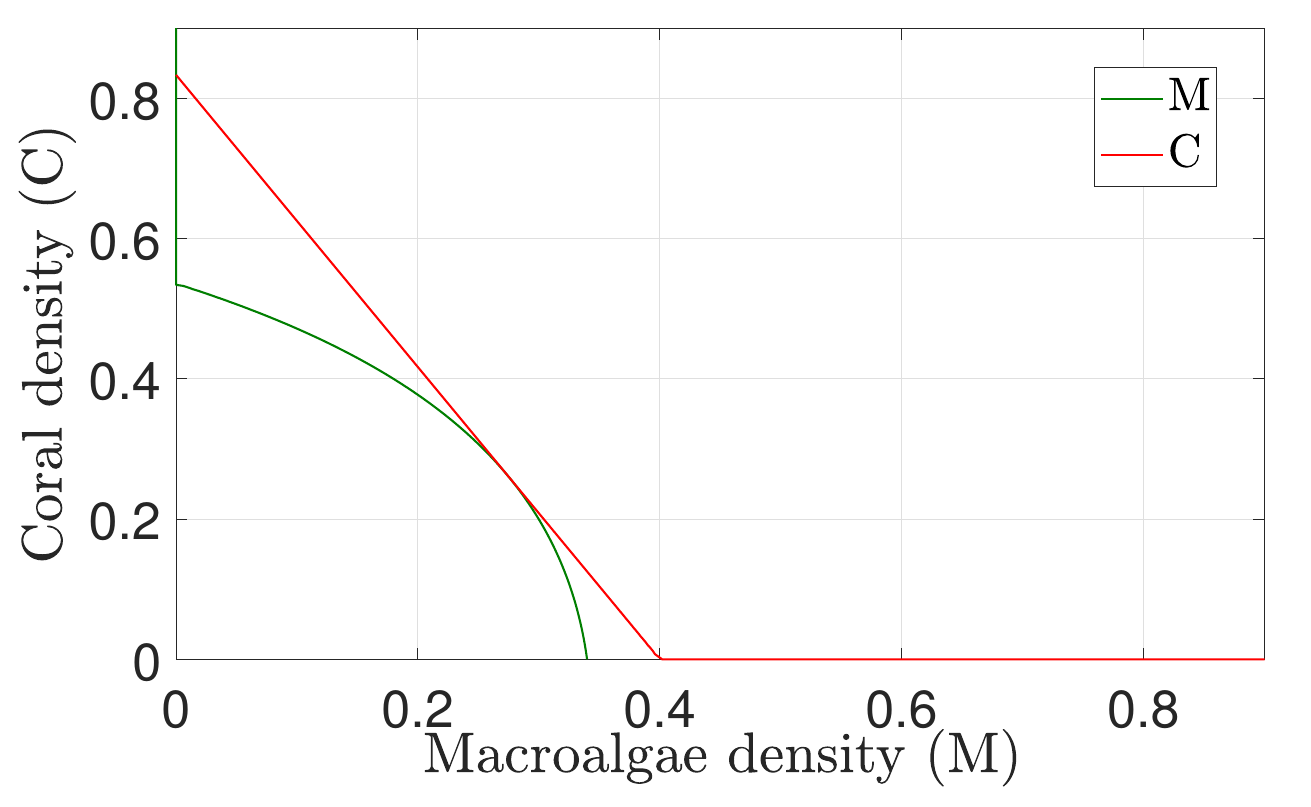}
		\caption{Nullclines indicates the occurrence of unique equilibrium point of model \eqref{paper6-mod2}.}
		\label{paper6-nullcline-0.112287}
	\end{center}
	\end{figure}
Depending on the parameters, the number of feasible interior equilibrium points for the proposed model is determined by comparing the relative positions and shapes of the nontrivial nullclines. Numerical analysis suggests that the potential number of coexistence equilibrium points ranges from zero to two, as depicted in Figures \ref{paper6-nullcline-0.2}-\ref{paper6-nullcline-0.05}, contingent upon the intrinsic growth rate $r$. Consider the fixed set of parameters $k=0.5,$ $a=0.65,$ $g=0.3,$ $\gamma=0.4,$ $d=0.1,$ and $\alpha=0.6.$ For $r= 0.2,$ the algae nullcline (green curve) and coral nullcline (red curve) cross twice, indicating the occurrence of two interior equilibrium (0.162603, 0.494576) and (0.369384, 0.063783) in the feasible region shown in Figure \ref{paper6-nullcline-0.2}. For $r=0.112287,$ the algae nullcline (green curve) and coral nullcline (red curve) touch each other and give unique equilibrium (0.27, 0.26) is shown in Figure \ref{paper6-nullcline-0.112287}. Two nullclines neither overlap nor contact each other for $r=0.05$, implying no equilibrium point, as shown in Figure \ref{paper6-nullcline-0.05}.
	\begin{figure}[h!]
		\begin{center}
		\includegraphics[width=10cm,height=5.5cm]{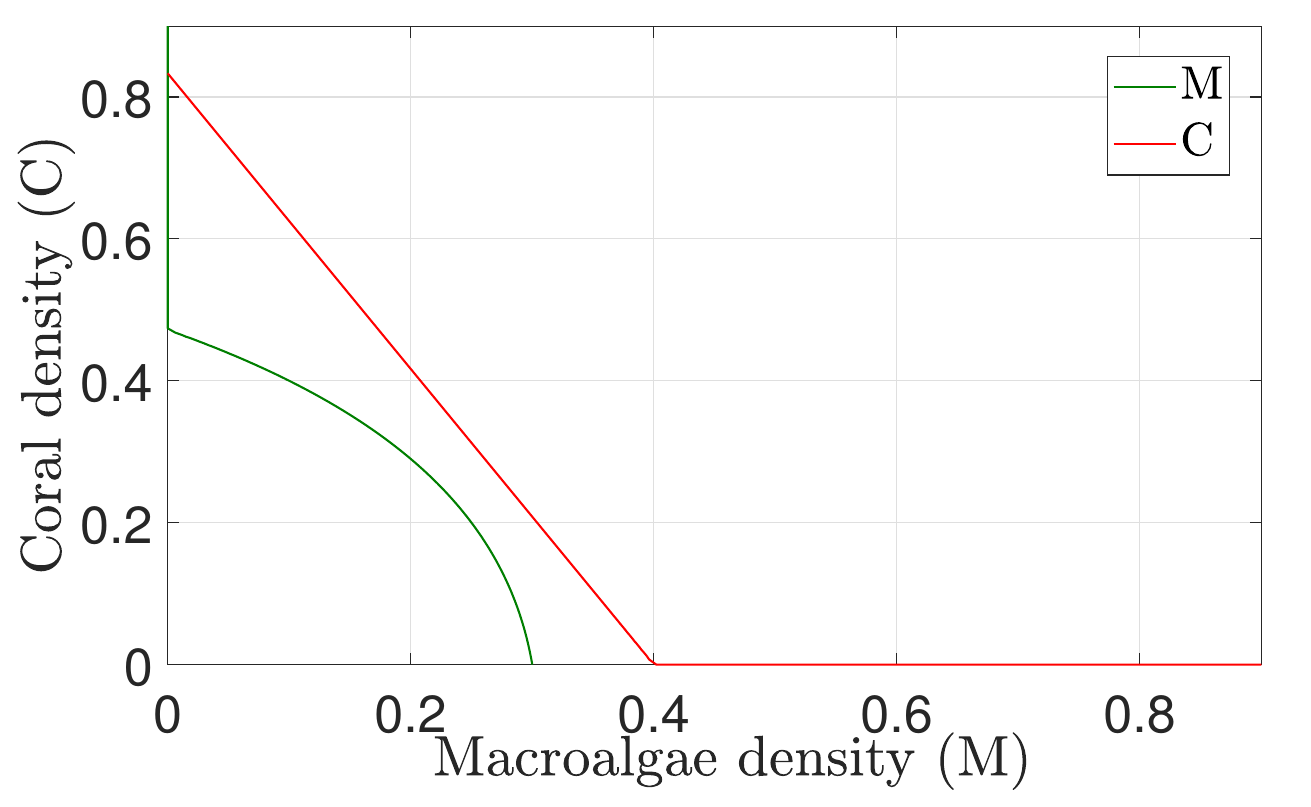}
		\caption{Nullclines indicates the occurrence of no equilibrium point of model \eqref{paper6-mod2}.}
		\label{paper6-nullcline-0.05}
	\end{center}
	\end{figure}

Considering the permanence of two species, here we concentrate only on the interior equilibrium $E^*(M^*, C^*)$. The eigenvalues of the Jacobian matrices corresponding to $E^*(M^*, C^*)$ are used to investigate the local stability of the proposed model \eqref{paper6-mod2}.  Consider the following Jacobian matrix for model \eqref{paper6-mod2} about $E^*(M^*, C^*)$:
\begin{align}\label{paper6-eq20}
J=J(M, C)=\begin{bmatrix}
a_{11} & a_{12}\\
a_{21} & a_{22}
\end{bmatrix},
\end{align}
where $a_{11}=1+\delta r-\frac{2 \delta r M}{k}+\delta a C-\frac{\delta g}{1-C}+\delta \gamma -2 \delta \gamma M-\delta \gamma C$,
$a_{12}=\delta a M-\frac{\delta g M}{(1-C)^2}-\delta \gamma M$,
$a_{21}=-\delta \alpha C-\delta a C$, 
and $a_{22}=1+\delta \alpha -\delta \alpha M-2 \delta \alpha C-\delta d-\delta a M.$

Let $F(\lambda)=\lambda^2+\mathcal{A}\lambda+\mathcal{B}$ where $\mathcal{A}=-\textrm{tr}(J)$ and $\mathcal{B}=\det (J)$ be the characteristic equation of associated Jacobian matrix $J$ evaluated at a fixed point $E^*(M^*, C^*)$. Then, the following Lemmas \ref{paper6-lem1} and \ref{paper6-lem2} may be demonstrated by the relationships between roots of $F(\lambda)=0$ and its coefficients. This will help us to examine the stability of interior equilibrium $E^*(M^*, C^*)$.

\begin{lemma}\label{paper6-lem1}
\cite{paper6-LuoAC} Suppose that $F(1)>0$, $\lambda_1$ and $\lambda_2$ are two characteristic roots of $F(\lambda)=0.$ Then
	\begin{itemize}
		\item[1.] $|\lambda_1|<1$ and $|\lambda_2|<1$ if and only if $F(-1)>0$ and $\mathcal{B}<1,$
		\item[2.] $|\lambda_1|<1$ and $|\lambda_2|>1$ (\textrm{or}~ $|\lambda_1|>1$ and $|\lambda_2|<1$) if and only if $F(-1)<0,$
		\item[3.] $|\lambda_1|>1$ and $|\lambda_2|>1$ if and only if $F(-1)>0$ and $\mathcal{B}>1,$
		\item[4.] $\lambda_1=-1$ and $|\lambda_2|\neq 1$ if and only if $F(-1)=0$ and $\mathcal{A}$ is neither 0 nor 2,
		\item[5.] $\lambda_1$ and $\lambda_2$ are complex and $|\lambda_1|=1$ and $|\lambda_2|=1$ if and only if $\mathcal{A}^2-4\mathcal{B}<0$ and $\mathcal{B}=1.$
	\end{itemize}
\end{lemma}

\begin{lemma}\label{paper6-lem2}
\cite{paper6-LuoAC} Suppose $\lambda_1$ and $\lambda_2$ be characteristics roots of $F(\lambda)=0$, then the interior equilibrium $E^*(M^*, C^*)$ of the model \eqref{paper6-mod2} is called 
	\begin{itemize}
		\item[1.] a sink if $|\lambda_1|<1$ and $|\lambda_2|<1$,
		\item[2.] a source if $|\lambda_1|>1$ and $|\lambda_2|>1$,
		\item[3.] a saddle if $|\lambda_1|>1$ and $|\lambda_2|<1$ (or $|\lambda_1|<1$ and $|\lambda_2|>1$),
		\item[4.] non-hyperbolic if either $|\lambda_1|=1$ or $|\lambda_2|=1.$
	\end{itemize}
\end{lemma}
\begin{remark}
To establish the necessary and sufficient conditions for the local stability of the model \eqref{paper6-mod2} in the following Theorem \ref{paper6-Th1}, it is essential to observe the requirements given in Lemmas \ref{paper6-lem1} and \ref{paper6-lem2}. The characteristic equation of the Jacobian matrix $J(M, C)$ is given by
\begin{align*}
\lambda^2+\mathcal{A}(M, C)\lambda+\mathcal{B}(M, C)=0,
\end{align*} 
where 
\begin{align*}
-\mathcal{A}(M, C)=a_{11}+a_{22}=2+\delta\mathcal{U},
\end{align*}
where $\mathcal{U}=u_1+u_2$ and 
\begin{align*}
\mathcal{B}(M, C)=a_{11}a_{22}-a_{12}a_{21}=1+\delta\mathcal{U}+\delta^2\mathcal{V},
\end{align*}
where $\mathcal{V}=u_1u_2-\frac{a_{12}a_{21}}{\delta^2}.$
Here, $u_1=r-\frac{2rM}{k}+aC-\frac{g}{1-C}+\gamma-2\gamma M-\gamma C$ and $u_2=\alpha-\alpha M-2 \alpha C-d-aM$.
Then 
\begin{align*}
F(\lambda)=&\lambda^2-\textrm{tr}(J)\lambda+\det(J)\\
=&\lambda^2-(2+\delta\mathcal{U})\lambda+(1+\delta\mathcal{U}+\delta^2\mathcal{V}).
\end{align*}

Here, $F(-1)=4+2\delta\mathcal{U}+\delta^2\mathcal{V}$ and $F(1)=\delta^2\mathcal{V}>0$. It is clear that from the assumption of Lemma \ref{paper6-lem1}, $F(1) >0$ if $\mathcal{V}>0$. Assume $\mathcal{V}>0.$
\end{remark}

\begin{theorem}\label{paper6-Th1}
	Let $E^*(M^*, C^*)$ be the  positive interior equilibrium and Lemma \ref{paper6-lem1} and \ref{paper6-lem2} holds for the model \eqref{paper6-mod2}. Then $E^*(M^*, C^*)$ is said to be
\begin{itemize}
\item[1.] a sink if one of the following conditions holds:
		\begin{itemize}
			\item[(i)] $\mathcal{U}=-2\sqrt{\mathcal{V}}$ and $0<\delta<-\frac{\mathcal{U}}{\mathcal{V}}$,
			\item[(ii)] $\mathcal{U}<-2\sqrt{\mathcal{V}}$ and $0<\delta<\frac{-\mathcal{U}-\sqrt{\mathcal{U}^2-4\mathcal{V}}}{\mathcal{V}}$.
		\end{itemize}
\item[2.] a source if one of the following conditions holds:
	\begin{itemize}
		\item[(i)] $\mathcal{U}=-2\sqrt{\mathcal{V}}$ and $\delta>-\frac{\mathcal{U}}{\mathcal{V}}$,
		\item[(ii)] $\mathcal{U}<-2\sqrt{\mathcal{V}}$ and $\delta>\frac{-\mathcal{U}+\sqrt{\mathcal{U}^2-4\mathcal{V}}}{\mathcal{V}}$,
		\item[(iii)] $\mathcal{U} \geq 0.$
	\end{itemize}
\item[3.] a saddle if the following conditions holds:
\begin{itemize}
	\item[] $\mathcal{U}<-2\sqrt{\mathcal{V}}$ and $\frac{-\mathcal{U}-\sqrt{\mathcal{U}^2-4\mathcal{V}}}{\mathcal{V}}<\delta<\frac{-\mathcal{U}+\sqrt{\mathcal{U}^2-4\mathcal{V}}}{\mathcal{V}}$. 
\end{itemize}
\item[4.] a non-hyperbolic if one of the following conditions holds:
\begin{itemize}
	\item[(i)] $\mathcal{U}<-2\sqrt{\mathcal{V}},$ $\delta=\frac{-\mathcal{U} \pm \sqrt{\mathcal{U}^2 -4\mathcal{V}}}{\mathcal{V}}$, and $\delta \neq -\frac{2}{\mathcal{U}}, -\frac{4}{\mathcal{V}}$,
	\item[(ii)] $-2\sqrt{\mathcal{V}}<\mathcal{U}<0$ and $\delta=-\frac{\mathcal{U}}{\mathcal{V}}.$
\end{itemize}
	\end{itemize}
\end{theorem}

\section{Bifurcation analysis}\label{paper6-sec3}
This section proposed to derive codimension-one Neimark-Sacker bifurcation around the interior equilibrium $E^*(M^*, C^*)$ for the model \eqref{paper6-mod2} where $\delta$ act as a bifurcation parameter.

\subsection{Flip bifurcation analysis around $E^*(M^*, C^*)$}

Let \begin{align*}
	\mathcal{F}_1=\left\{ \Big(r, k, a, g, \gamma, \alpha, d, S, \delta\Big): \delta=\frac{-\mathcal{U}-\sqrt{\mathcal{U}^2-4\mathcal{V}}}{\mathcal{V}}, \mathcal{U}<-2\sqrt{\mathcal{V}}, r, k, a, g, \gamma, \alpha, d, S, \delta>0 \right\},
\end{align*}
or 
\begin{align*}
	\mathcal{F}_2=\left\{ \Big(r, k, a, g, \gamma, \alpha, d, S, \delta\Big): \delta=\frac{-\mathcal{U}+\sqrt{\mathcal{U}^2-4\mathcal{V}}}{\mathcal{V}}, \mathcal{U}<-2\sqrt{\mathcal{V}}, r, k, a, g, \gamma, \alpha, d, S, \delta>0 \right\}.
\end{align*}

Now we analyze, the flip bifurcation of $E^*(M^*, C^*)$ if the parameters fluctuate in a restricted area around $\mathcal{F}_1$ (or $\mathcal{F}_2$).

We begin with a discussion about the flip bifurcation of model \eqref{paper6-mod2} at $E^*(M^*, C^*)$ when parameters fluctuate in a restricted area around $\mathcal{F}_1$. The other instance $\mathcal{F}_2$ can be justified by the same reasoning. Consider the model \eqref{paper6-mod2} with $(r, k, a, g, \gamma, \alpha, d, S, \delta_1) \in \mathcal{F}_1$, which is given by
\begin{align}\label{paper6-eq1}
	M &\longrightarrow M+\delta_1\left[rM\left(1-\frac{M}{k}\right)+aMC-\frac{gM}{M+S}+\gamma MS \right],\nonumber\\
	C &\longrightarrow C+\delta_1 \Big[\alpha SC-dC-aMC\Big].
\end{align}

The interior equilibrium point $E^*(M^*, C^*)$ of map \eqref{paper6-eq1}, whose eigenvalues are $\lambda_1=-1$, and $\lambda_2=3+\mathcal{U}\delta_1$ with $|\lambda_2| \neq 1$ by Theorem \ref{paper6-Th1}. 

Since $(r, k, a, g, \gamma, \alpha, d, S, \delta_1) \in \mathcal{F}_1$ and $\delta_1=\frac{-\mathcal{U}-\sqrt{\mathcal{U}^2-4\mathcal{V}}}{\mathcal{V}}$, we examine a perturbation of \eqref{paper6-eq1} using $\delta_1^*$ as the bifurcation parameter as follows:
\begin{align}\label{paper6-eq2}
	M &\longrightarrow M+(\delta_1+\delta_1^*)\left[rM\left(1-\frac{M}{k}\right)+aMC-\frac{gM}{M+S}+\gamma MS \right],\nonumber\\
	C &\longrightarrow C+(\delta_1+\delta_1^*) \Big[\alpha SC-dC-aMC\Big],
\end{align}
where $|\delta_1^*|\ll 1,$ is a small perturbation parameter around $\delta_1$.
\begin{theorem}\label{paper6-Th2}
	The interior equilibrium point $E^*(M^*, C^*)$ of model \eqref{paper6-mod2} experience flip bifurcation when $\delta$ changes in a small neighborhood of $\delta_1$ if $\Omega_1 \neq 0$ and $\Omega_2 \neq 0$ which is defined in Appendix A. Furthermore, period-2 orbits bifurcates from $E^*(M^*, C^*)$ is stable(unstable) if $\Omega_2>0(<0)$.
\end{theorem}
\begin{proof}
	The Proof of Theorem \ref{paper6-Th2} is given in Appendix A. 
\end{proof}
\begin{remark}
	In discrete-time models, taking integral step size $\delta=\delta_2=\frac{-\mathcal{U}-\sqrt{\mathcal{U}^2-4\mathcal{V}}}{\mathcal{V}}$ (which depends on model's parameter, see the statement of Theorem \ref{paper6-Th1}) as a bifurcation point can affect the stability and behavior of the model. When the integral step size is small, the model exhibits stable behavior. However, the model may exhibit complex behaviors such as chaos as the integral step size increases. The value of $\delta$ can significantly impact the model's behavior, especially near bifurcation points. Changing $\delta$ can alter the relative strengths of different terms in the model equations, potentially triggering a qualitative shift in the system's dynamics (bifurcation). Therefore, the integral step size is an important parameter to consider when analyzing the behavior of the discrete-time model. For this reason, the authors\cite{paper6-Salman,paper6-He,paper6-Elabbasy,paper6-Liu,paper6-Abdelaziz} analyzed the behavior of the dynamical model by considering integral step size, which depends on model parameters as a bifurcation parameter. 
\end{remark}
\subsection{Neimark-Sacker bifurcation analysis around $E^*(M^*, C^*)$}
Consider,
\begin{align*}
	\mathcal{N}=\Bigg\{\Big(r, k, a, g, \gamma, \alpha, d, \delta \Big): \delta=-\frac{\mathcal{U}}{\mathcal{V}}, -2\sqrt{\mathcal{V}}<\mathcal{U}<0,r, k, a, g, \gamma, \alpha, d, \delta>0 \Bigg\}.
\end{align*}

When parameters fluctuate in a small neighborhood of $\mathcal{N}$, Neimark-Sacker bifurcation may occur at $E^*(M^*, C^*)$.

The Neimark-Sacker bifurcation of $E^*(M^*, C^*)$ is discussed if the parameters $(r, k, a, g, \gamma, \alpha, d, \delta)$ vary within a small neighborhood of $\mathcal{N}$. We take parameters $(r, k, a, g, \gamma, \alpha, d, \delta_2)$ arbitrarily from $\mathcal{N}$ then the model is given by
\begin{align}\label{paper6-eq6}
	x &\longrightarrow x+\delta_2\left[rM\left(1-\frac{M}{k}\right)+aMC-\frac{gM}{1-C}+\gamma M(1-M-C) \right],\nonumber\\
	y &\longrightarrow y+\delta_2\Big[r(1-M-C)C-dC-aMC \Big].
\end{align}

Given that $(r, k, a, g, \gamma, \alpha, d, \delta_2) \in \mathcal{N}, \delta_2=-\frac{\mathcal{U}}{\mathcal{V}}$. Taking $\delta_2^*$ as the bifurcation parameter, we take the following perturbation of \eqref{paper6-eq6}:
\begin{align}\label{paper6-eq2}
x &\longrightarrow x+(\delta_2+\delta_2^*)\left[ rM\left(1-\frac{M}{k}\right)+aMC-\frac{gM}{1-C}+\gamma M(1-M-C) \right],\nonumber\\
y &\longrightarrow y+(\delta_2+\delta_2^*)\Big[ r(1-M-C)C-dC-aMC \Big],
\end{align}
where $|\delta_2^*|\ll 1,$ is a small perturbation parameter. 

\begin{theorem}\label{paper6-Th3}
	The interior equilibrium $E^*(M^*, C^*)$ of the model \eqref{paper6-mod2} experiences Neimark-Sacker bifurcation when $\delta$ changes in a small neighborhood of $\delta_2$ if $\Psi \neq 0$ which is defined in Appendix B. Furthermore, an attracting(repelling) invariant closed curve bifurcates from $E^*(M^*, C^*)$ for $\delta>\delta_2(<\delta_2)$ if $\Psi<0(>0)$.
\end{theorem}
\begin{proof}
	The Proof of Theorem \ref{paper6-Th3} is given in Appendix B. 
\end{proof}
\section{Chaos control}\label{paper6-sec4}
\subsection{OGY chaos control method}
Controlling chaos and bifurcation are believed to be crucial characteristics of interacting populations. Discrete-time models often exhibit more complex behavior than continuous ones. Chaos management techniques are required to protect the population from unpredictable events. In this part, we apply the OGY approach \cite{paper6-Ott} to the model for guiding an unstable trajectory toward a stable one. Consider model \eqref{paper6-mod11} by taking the integral step size $\delta$ as a control parameter and applying the OGY approach to find the stable region around the neighborhood of $\delta$
\begin{align}\label{paper6-eq9}
	M_{n+1}&=M_n+\delta\left[r M_n\left(1-\frac{M_n}{k}\right)+aM_n C_n-\frac{g M_n}{1-C_n}+\gamma M_n (1-M_n-C_n)  \right]=f_{11}(M_n, C_n, \delta),\nonumber\\
	C_{n+1}&=C_n+\delta\Big[r (1-M_n-C_n)C_n-dC_n-aM_nC_n \Big]=f_{12}(M_n, C_n, \delta).
\end{align}

Here, $\delta$ acts as a control parameter;  very small perturbations in $\delta$ can accomplish the desired chaos control. To do this, we constrain $\delta$ to lie in a small interval $\delta \in (\delta_0-\epsilon, \delta_0+\epsilon )$, $\epsilon>0$ where $\delta_0$ implies that the nominal value belongs to the chaotic region.

We use the stabilizing feedback control approach to direct the trajectory toward the intended orbit. The following linear map can be used to approximate model \eqref{paper6-eq9} in the neighborhood of the unstable equilibrium $E^*(M^*, C^*)$ of the model \eqref{paper6-mod2} in the chaotic zone created by the formation of Neimark-Sacker bifurcation.
\begin{align}\label{paper6-eq10}
	\begin{bmatrix}
	M_{n+1}-M^*\\
	C_{n+1}-C^*
	\end{bmatrix}\approx J(M^*, C^*, \delta_0) \begin{bmatrix}
	M_n-M^*\\
	C_n-C^*
	\end{bmatrix}+B[\delta-\delta_0],
\end{align}
where
\begin{align*}
	J(M^*, C^*, \delta_0)&=\begin{bmatrix}
	\frac{\partial f_{11}}{\partial M} & \frac{\partial f_{11}}{\partial C}\\
	\frac{\partial f_{12}}{\partial M} & \frac{\partial f_{12}}{\partial C}
	\end{bmatrix}=\begin{bmatrix}
	\hat{a} & \hat{b}\\
	\hat{c} & \hat{d}
	\end{bmatrix}\\
	&=\begin{bmatrix}
	1\!+\!\delta r\!-\!\frac{2 \delta r M}{k}\!+\!\delta a C\!-\!\frac{\delta g}{1-C}\!+\!\delta \gamma\!-\!2\delta\gamma M \!-\!\delta\gamma C  & \delta a M\!-\!\frac{\delta g M}{(1-C)^2}\!-\!\delta \gamma M\\
	-\delta \alpha C\!-\!\delta a C & 1\!+\!\delta \alpha \!-\!\delta \alpha M\!-\!2 \delta \alpha C\!-\!\delta d\!-\!\delta a M
	\end{bmatrix}
\end{align*}
and
\begin{align*}
	B=\begin{bmatrix}
	\frac{\partial f_{11}}{\partial \delta}\\
	\frac{\partial f_{12}}{\partial \delta}
	\end{bmatrix}=\begin{bmatrix}
	rM\left(1-\frac{M}{k}\right)+aMC-\frac{gM}{1-C}+\gamma M(1-M-C) \\
	r(1-M-C)C-dC-aMC
	\end{bmatrix}=\begin{bmatrix}
	\tilde{a}\\
	\tilde{b}
	\end{bmatrix}.
\end{align*}

Additionally, model \eqref{paper6-eq9} is controllable if $C$ is a rank 2 matrix,
\begin{align*}
	C=[B:JB]=\begin{bmatrix}
	\tilde{a} & \hat{a}\tilde{a}+\hat{b} \tilde{b}\\
	\tilde{b} & \hat{c}\tilde{a}+\hat{d}\tilde{b}
	\end{bmatrix}.
\end{align*}

Take
\begin{align*}
	[\delta-\delta_0]=-H\begin{bmatrix}
	M_n-M^*\\
	C_n-C^*
	\end{bmatrix},
\end{align*}
where $H=[\rho_1~ \rho_2]$ is a gain matrix. Then the model \eqref{paper6-eq10} is given by
\begin{align*}
\begin{bmatrix}
M_{n+1}-M^*\\
C_{n+1}-C^*
\end{bmatrix}\approx [J-BH]\begin{bmatrix}
M_n-M^*\\
C_n-C^*
\end{bmatrix}.	
\end{align*}

Furthermore, $E^*(M^*, C^*)$ is locally asymptotically stable if and only if all its eigenvalues lie in the unit disk. Consider $J-BH$
\begin{align*}
	J-BH=\begin{bmatrix}
	\hat{a} & \hat{b}\\
	\hat{c} & \hat{d}
	\end{bmatrix}-\begin{bmatrix}
	\tilde{a}\rho_1 & \tilde{a}\rho_2\\
	\tilde{b}\rho_1 & \tilde{b}\rho_2
	\end{bmatrix}=\begin{bmatrix}
	\hat{a}-\tilde{a}\rho_1 & \hat{b}-\tilde{a}\rho_2\\
	\hat{c}-\tilde{b}\rho_1 & \hat{d}-\tilde{b}\rho_2
	\end{bmatrix}.
\end{align*}

The characteristic equation $J-BH$ is given by 
\begin{align*}
	P(\lambda)=\lambda^2-(\lambda_1+\lambda_2)\lambda+(\lambda_1 \lambda_2),
\end{align*}
where
\begin{align}
	\lambda_1+\lambda_2&=\hat{a}+\hat{d}-\tilde{a}\rho_1-\tilde{b}\rho_2,\label{paper6-eq11}\\
	\lambda_1\lambda_2&=(\hat{a}-\tilde{a}\rho_1)(\hat{d}-\tilde{b}\rho_2)-(\hat{c}-\tilde{b}\rho_1)(\hat{b}-\tilde{a}\rho_2)\label{paper6-eq12}.
\end{align}

Next, to determine the lines of marginal stability for the associated controlled model, we consider $\lambda_1=\pm 1$ and $\lambda_1\lambda_2=1$. Additionally, these limitations guarantee that the open unit disk contains $\lambda_1$ and $\lambda_2$. Assume $\lambda_1\lambda_2=1$, then \eqref{paper6-eq12} became
\begin{align*}
	\mathcal{L}_1: (\hat{a}-\tilde{a}\rho_1)(\hat{d}-\tilde{b}\rho_2)-(\hat{c}-\tilde{b}\rho_1)(\hat{b}-\tilde{a}\rho_2)-1=0.
\end{align*}

Next, assume that $\lambda_1=1$, then from \eqref{paper6-eq11} and \eqref{paper6-eq12} we have
\begin{align*}
	\mathcal{L}_2: \hat{a}+\hat{d}-\tilde{a}\rho_1-\tilde{b}\rho_2-1=\hat{a}\hat{d}-\hat{a}\tilde{b}\rho_2-\tilde{a}\hat{d}\rho_1+\tilde{a}\tilde{b}\rho_1\rho_2-\hat{c}\hat{b}+\tilde{a}\hat{c}\rho_2+\tilde{d}\hat{b}\rho_1-\tilde{a}\tilde{b}\rho_1\rho_2.
\end{align*}

Finally assume $\lambda_2=-1$, then from \eqref{paper6-eq11} and \eqref{paper6-eq12} we have
\begin{align*}
\mathcal{L}_3: \hat{a}+\hat{d}-\tilde{a}\rho_1-\tilde{b}\rho_2+1=-[\hat{a}\hat{d}-\hat{a}\tilde{b}\rho_2-\tilde{a}\hat{d}\rho_1+\tilde{a}\tilde{b}\rho_1\rho_2-\hat{c}\hat{b}+\tilde{a}\hat{c}\rho_2+\tilde{d}\hat{b}\rho_1-\tilde{a}\tilde{b}\rho_1\rho_2].
\end{align*}

Thus, the bounded region obtained by the intersection of lines $\mathcal{L}_1$, $\mathcal{L}_2$, and $\mathcal{L}_3$ gives the stable eigenvalues for the model \eqref{paper6-mod2}.  

\subsection{Feedback control method}
We use the hybrid control feedback methodology \cite{paper6-Luo}, which was mainly developed for controlling the period-doubling bifurcation; a similar technique is used in \cite{paper6-Yuan} to manage the chaos caused by the appearance of Neimark-Sacker bifurcation. Assuming model \eqref{paper6-mod2} experiences Neimark-Sacker bifurcation at equilibrium $E^*(M^*, C^*)$, the associated controlled model is given by
\begin{align}\label{paper6-eq13}
	M_{n+1}=\zeta f_{11}(M_n, C_n, \delta)+(1-\zeta)M_n,\nonumber\\
	C_{n+1}=\zeta f_{12}(M_n, C_n, \delta)+(1-\zeta)C_n,
\end{align}
where $0<\zeta<1.$ The controlled method in \eqref{paper6-eq13} combines feedback control $\zeta$ and parameter perturbation $\delta$. Additionally, the Neimark-Sacker bifurcation of the equilibrium $E^*(M^*, C^*)$ of the controlled model \eqref{paper6-eq13} can be delayed, or perhaps completely removed by making an appropriate choice of the controlled parameter. From the Jacobian matrix \eqref{paper6-eq20}, we write the corresponding Jacobian  $J^*$ for the controlled model \eqref{paper6-eq13}
\begin{align*}
	J^*=\begin{bmatrix}
	\zeta a_{11}+(1-\zeta) & \zeta a_{12}\\
	\zeta a_{21} & \zeta a_{22}+(1-\zeta)
	\end{bmatrix}.
\end{align*}

Then the characteristic equation is given by
\begin{align*}
	\lambda^2-\textrm{tr}(J^*)\lambda+\det(J^*)=0,
\end{align*}
where \begin{align*}
	\textrm{tr}(J^*)&=\zeta a_{11}+\zeta a_{22}+2(1-\zeta),\\
	\det(J^*)&=\Big(\zeta a_{11}+(1-\zeta)\Big)\Big(\zeta a_{22}+(1-\zeta)\Big)-\zeta^2 a_{12}a_{21}
\end{align*}
where $a_{11},$ $a_{12},$ $a_{21},$ and $a_{22}$ are given previously. Using \cite{paper6-Yuan}, the following result identifies the requisites for the positive equilibrium $E^*(M^*, C^*)$ of the controlled model \eqref{paper6-eq13} to be locally asymptotically stable.
\begin{theorem}
If the following condition holds,
\begin{align*}
|\textrm{tr} (J^*)|-1<\textrm{det}(J^*)<1,
\end{align*}
then the positive equilibrium $E^*(M^*, C^*)$ of the controlled model \eqref{paper6-eq13} is locally asymptotically stable.		
\end{theorem}
\section{Results and discussions}\label{paper6-sec5}
In this section, to verify the aforementioned theoretical analysis numerically by exploiting complicated dynamical behaviors such as bifurcation diagrams and phase portraits for the model \eqref{paper6-mod2}. Particularly the Macroalgae and coral reef populations dynamical behaviors are analyzed in marine ecosystems under intrinsic growth rate.

\begin{figure}[h]\textit{}
	\centering
	\begin{subfigure}
		\centering
		\includegraphics[width=7.5cm,height=5cm]{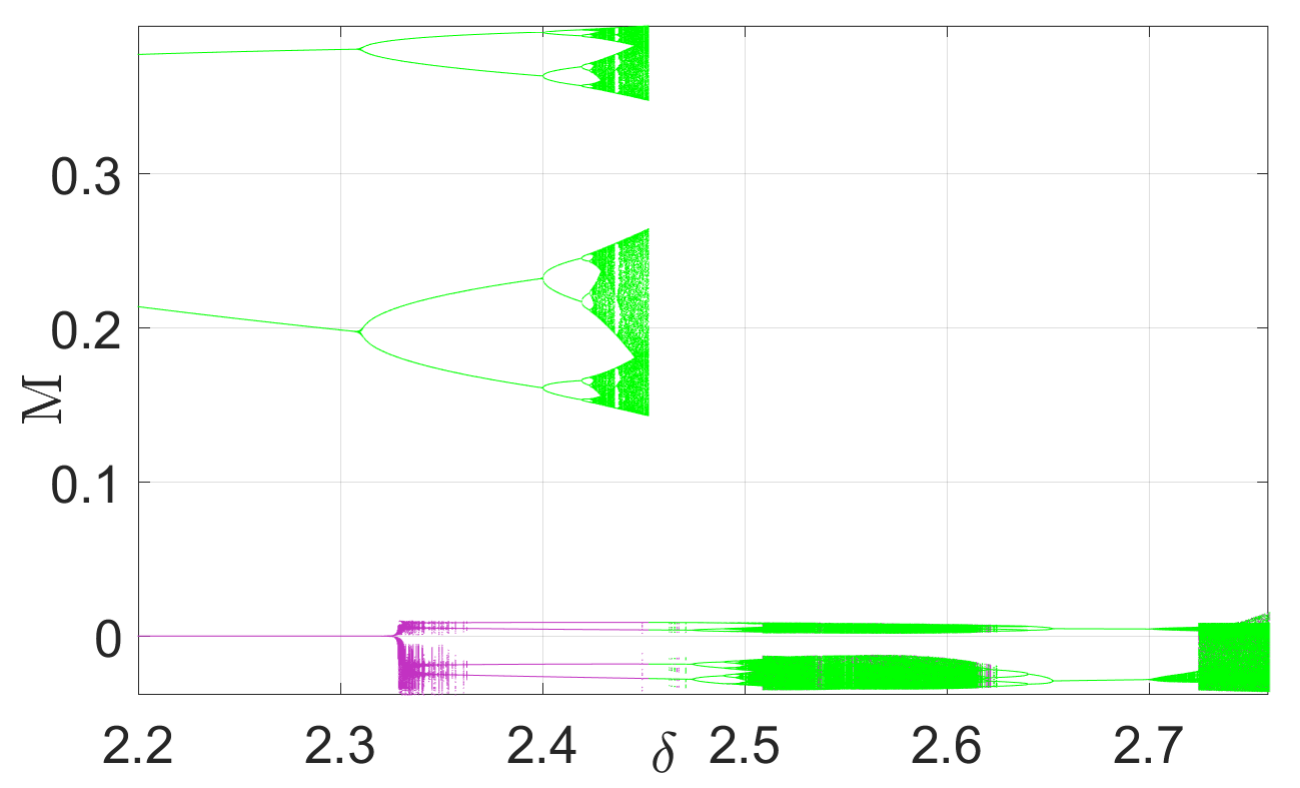}
	\end{subfigure}
	\begin{subfigure}
		\centering
		\includegraphics[width=7.5cm,height=5cm]{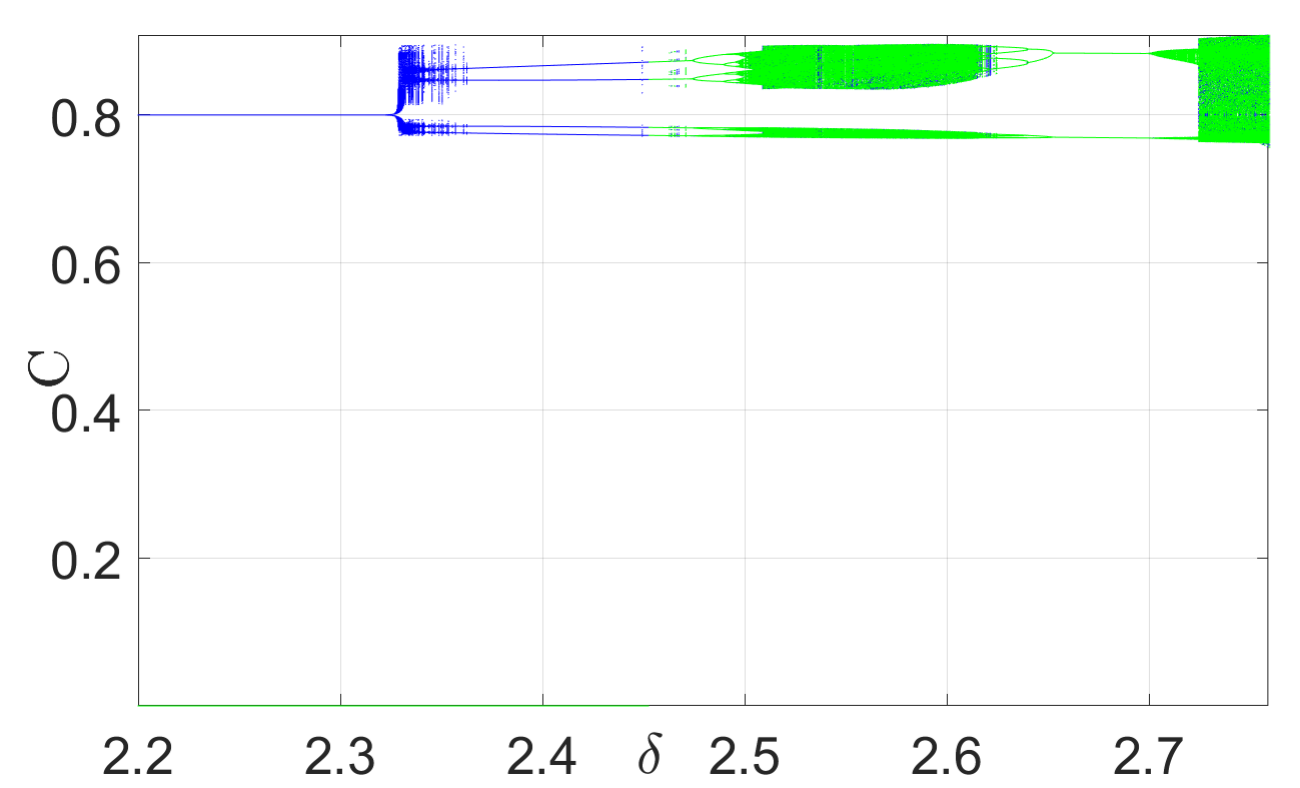}
	\end{subfigure}
	\caption{Reverse period-doubling bifurcation for $\delta$ versus $M$ and $C$.}
	\label{paper6-bif-r-0.76}
\end{figure}

Consider the following set of parametric value for the proposed model \eqref{paper6-mod2}, $k=0.3,$ $a=0.9,$ $g=0.5,$ $\gamma=0.8,$ $d=0.1,$ $\alpha=0.5$ with $\delta$ varies in range $[2.2, 2.75]$ and take initial condition as $(M_0, C_0)=(0.04, 0.66)$. The model \eqref{paper6-mod2} endures full Feigenbaum remerging tree for $r=0.76$. The term “full Feigenbaum three” refers to the situation where the sequence of period-doubling leading to chaos is followed by the reverse process when a parameter varies monotonously. Reverse period-doubling bifurcation diagrams for $M$ and $C$ for the proposed macroalgae-coral reef model with macroalgae intrinsic growth rate $r=0.76$ with other parameter fixed given previously is given in Figure \ref{paper6-bif-r-0.76}. 

\begin{figure}[h]\textit{}
	\centering
	\begin{subfigure}
		\centering
		\includegraphics[width=7.5cm,height=5cm]{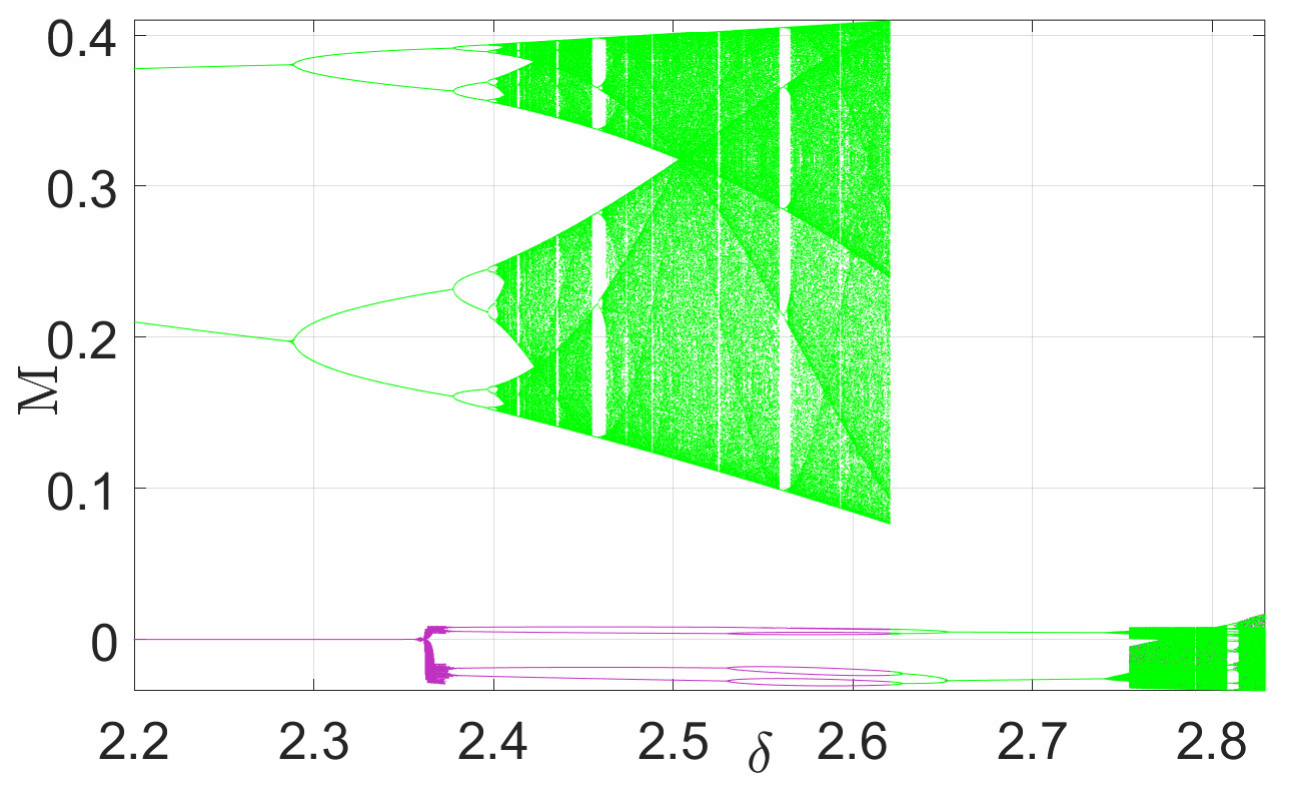}
	\end{subfigure}
	\begin{subfigure}
		\centering
		\includegraphics[width=7.5cm,height=5cm]{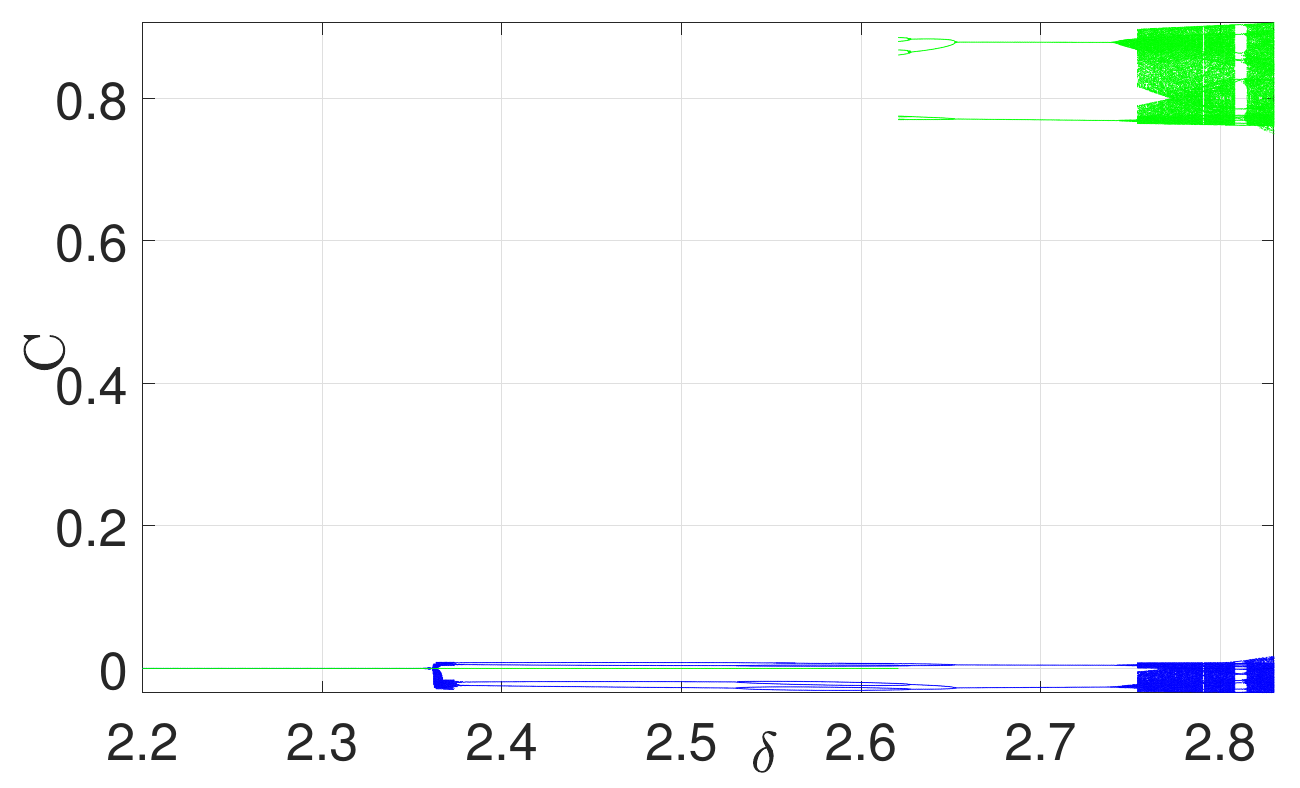}
	\end{subfigure}
	\caption{Period-8 bubble bifurcation for $\delta$ versus $M$ and $C$.}
	\label{paper6-bif-r-0.77}
\end{figure}

\begin{figure}[h]\textit{}
	\centering
	\begin{subfigure}
		\centering
		\includegraphics[width=7.5cm,height=5cm]{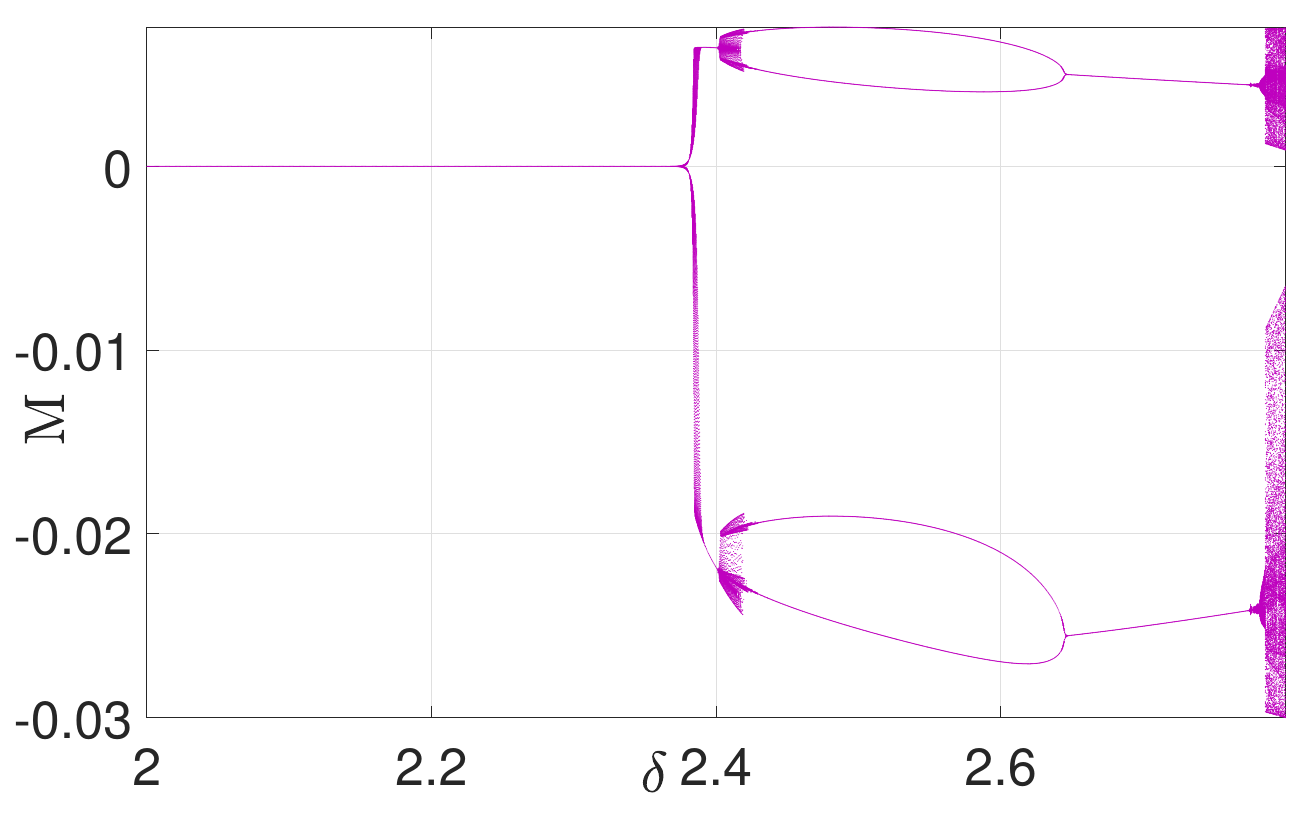}
	\end{subfigure}
	\begin{subfigure}
		\centering
		\includegraphics[width=7.5cm,height=5cm]{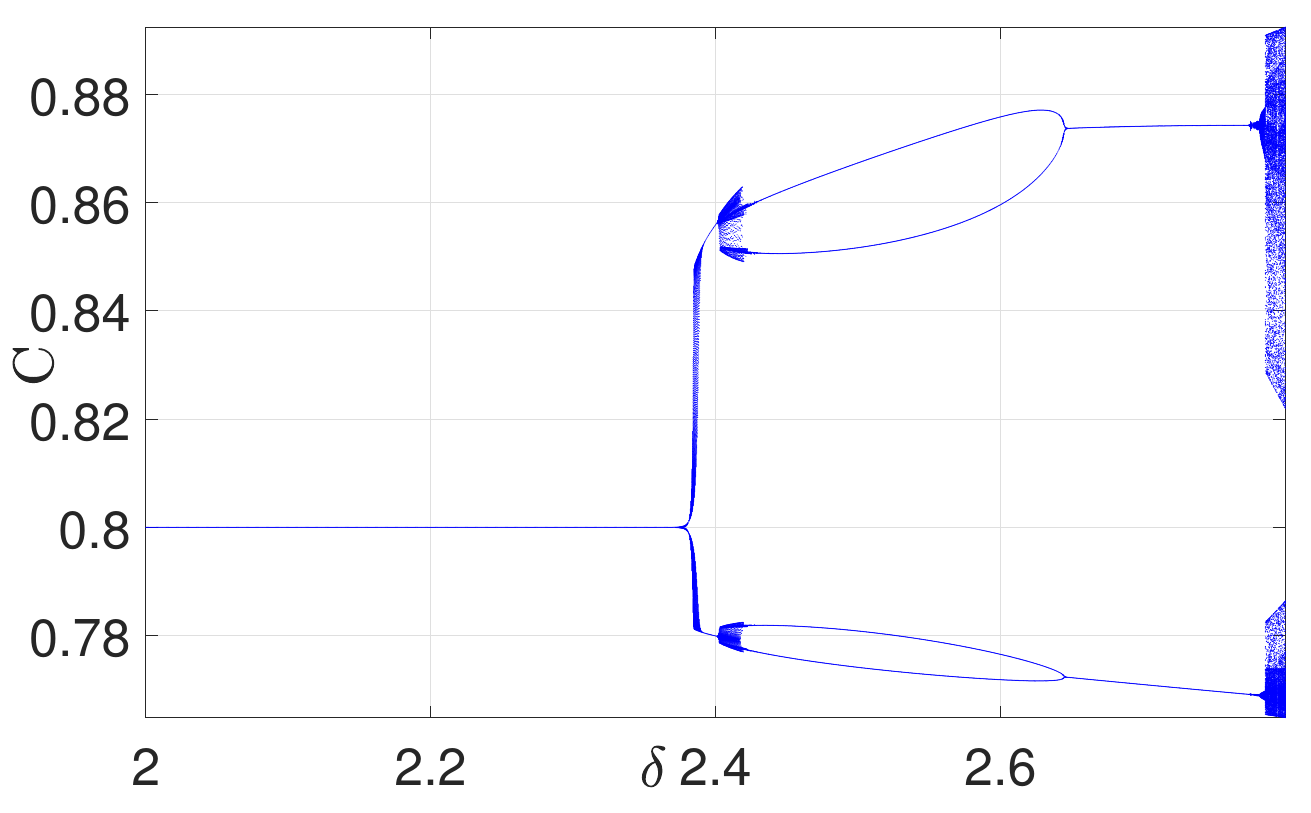}
	\end{subfigure}
	\caption{Period-4 bubble bifurcation for $\delta$ versus $M$ and $C$.}
	\label{paper6-bif-r-0.78}
\end{figure}

If we change the intrinsic growth rate as $r=0.77$ with same fixed set of parameters and $\delta$ varies in the range $[2.2, 2.8]$, the model \eqref{paper6-mod2} creates a primary period-8 bubble which is given in Figure \ref{paper6-bif-r-0.77}. The bifurcation diagram for the model with the same fixed set of parameters is shown with two different initial conditions: (0.04, 0.66) plotted in violet and (0.035, 0.59) plotted in green. As depicted in Figure \ref{paper6-bif-r-0.76} and \ref{paper6-bif-r-0.77}, it is evident that the model exhibits multistability. Period-4 bubble is generated for $r=0.78$ with $\delta$ varies in the range $=[2.2, 2.8]$ which is given in Figure \ref{paper6-bif-r-0.78}. Similarly, a period-24 bubble is developed for $r=0.8$ which is given in Figure \ref{paper6-bif-r-0.8}. Moreover interesting and exciting bifurcation exist for $r=0.75$ and $r=0.78$ which consists of several reversal period-doubling bifurcation which is given Figures \ref{paper6-bif-r-0.795} and \ref{paper6-bif-r-0.798}. When further increasing the values to $r=1$, the model undergoes period-doubling bifurcation  which is given in Figure \ref{paper6-bif-r-1}.

\begin{remark}
	The growth rate of macroalgae $r$ has a significant impact on the dynamics of the model even if we merely change the intrinsic growth rate and leave the other parameters fixed. 
\end{remark}

We provided the phase portrait for the above set of parameters, with $\delta$ fixed at $2.8$, and $r$ varied as follows: (a) $r=0.795$, (b) $r=0.78$, (c) $r=0.77$, and (d) $r=0.765$ (see Figure \ref{Paper6-pp-0.795}). The figures depict chaotic attractors with varying values of $r$. Consider the following set of parameters: $r=0.5$, $k=0.7$, $a=0.65$, $g=0.3$, $\gamma=0.4$, $d=0.1$, $\alpha=0.6$, and $\delta=3.3$, with different initial conditions: (e) (0.005, 0.6), (f) (0.004, 0.6), and (g) (0.001, 0.6) (see Figure \ref{Paper6-pp-0.795}). The figures depict coexisting attracting orbits.

\begin{figure}[h]\textit{}
	\centering
	\begin{subfigure}
		\centering
		\includegraphics[width=7.5cm,height=5cm]{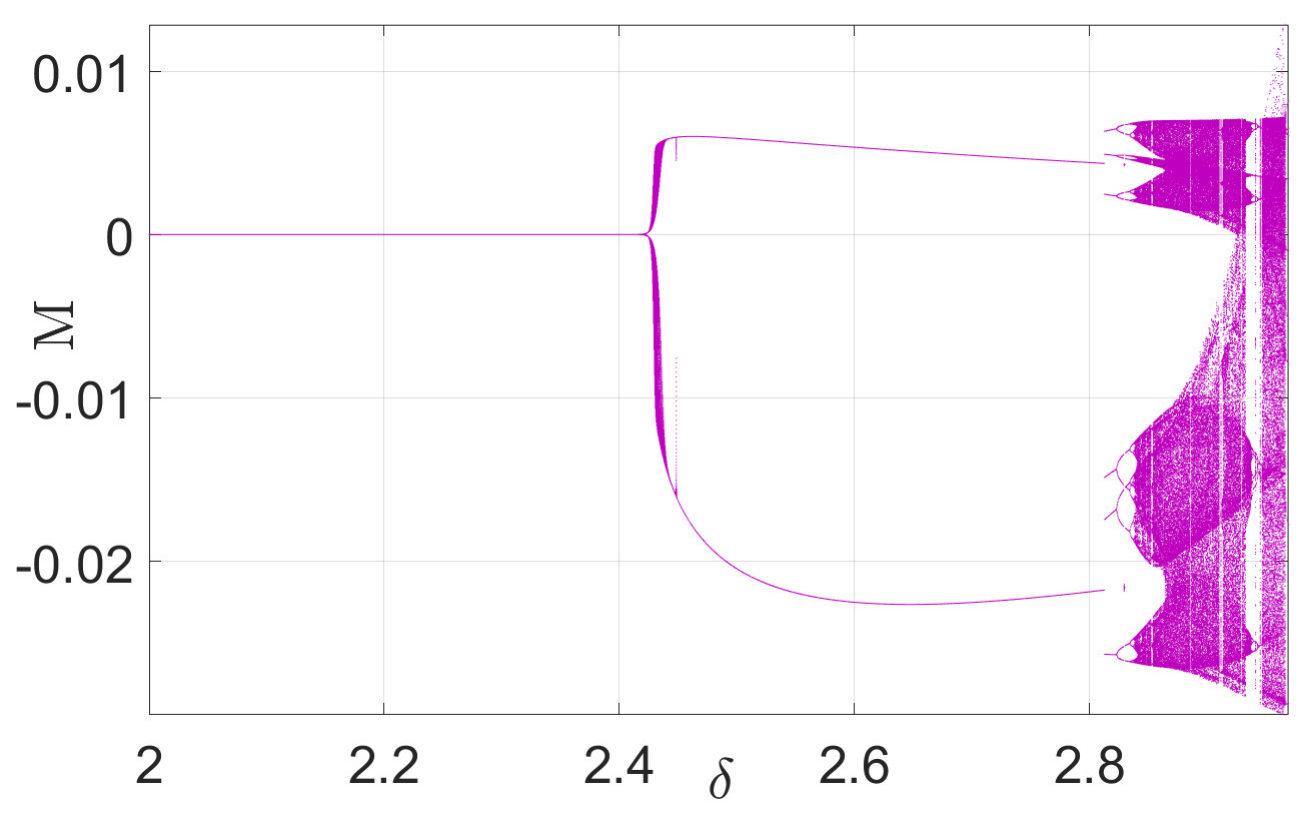}
	\end{subfigure}
	\begin{subfigure}
		\centering
		\includegraphics[width=7.5cm,height=5cm]{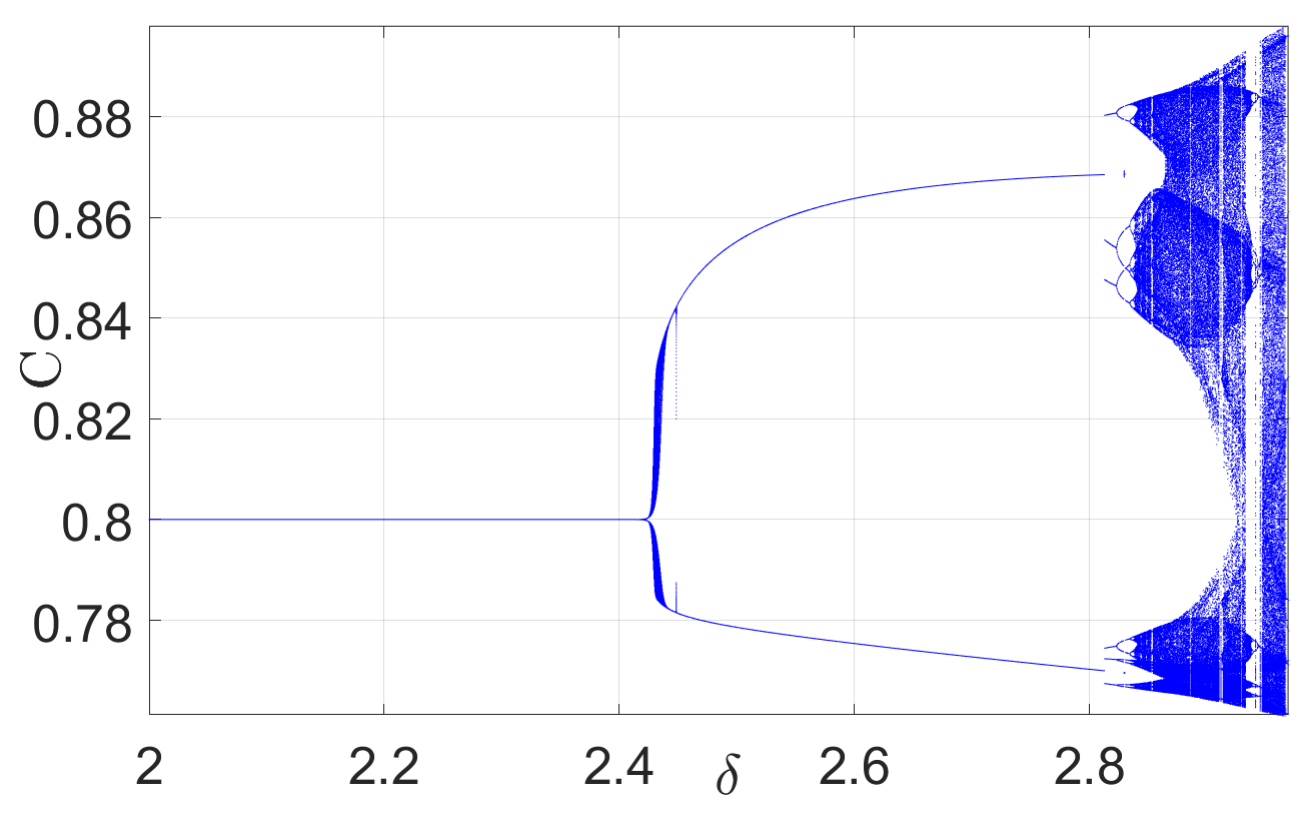}
	\end{subfigure}
	\caption{Bifurcation diagram for $r=0.795$.}
	\label{paper6-bif-r-0.795}
\end{figure}

\begin{figure}[h]\textit{}
	\centering
	\begin{subfigure}
		\centering
		\includegraphics[width=7.5cm,height=5cm]{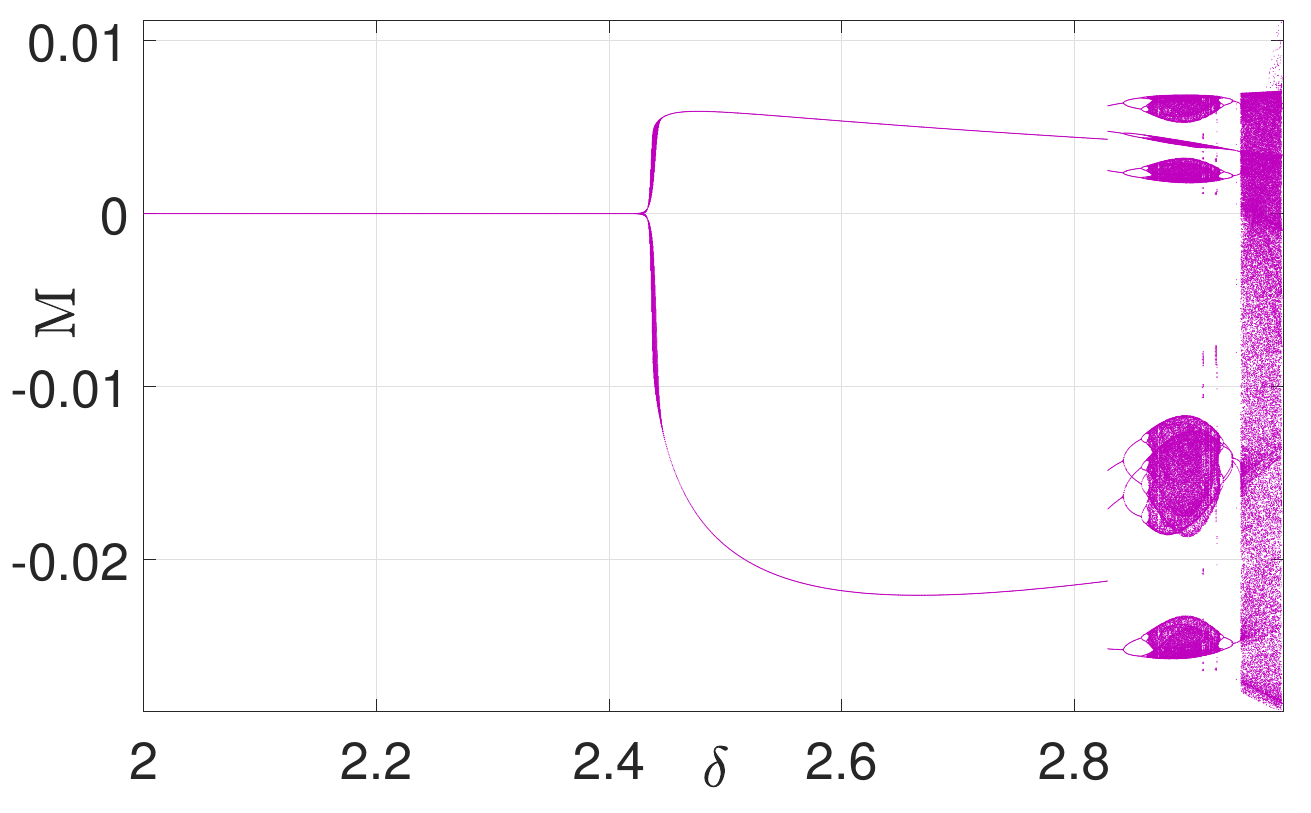}
	\end{subfigure}
	\begin{subfigure}
		\centering
		\includegraphics[width=7.5cm,height=5cm]{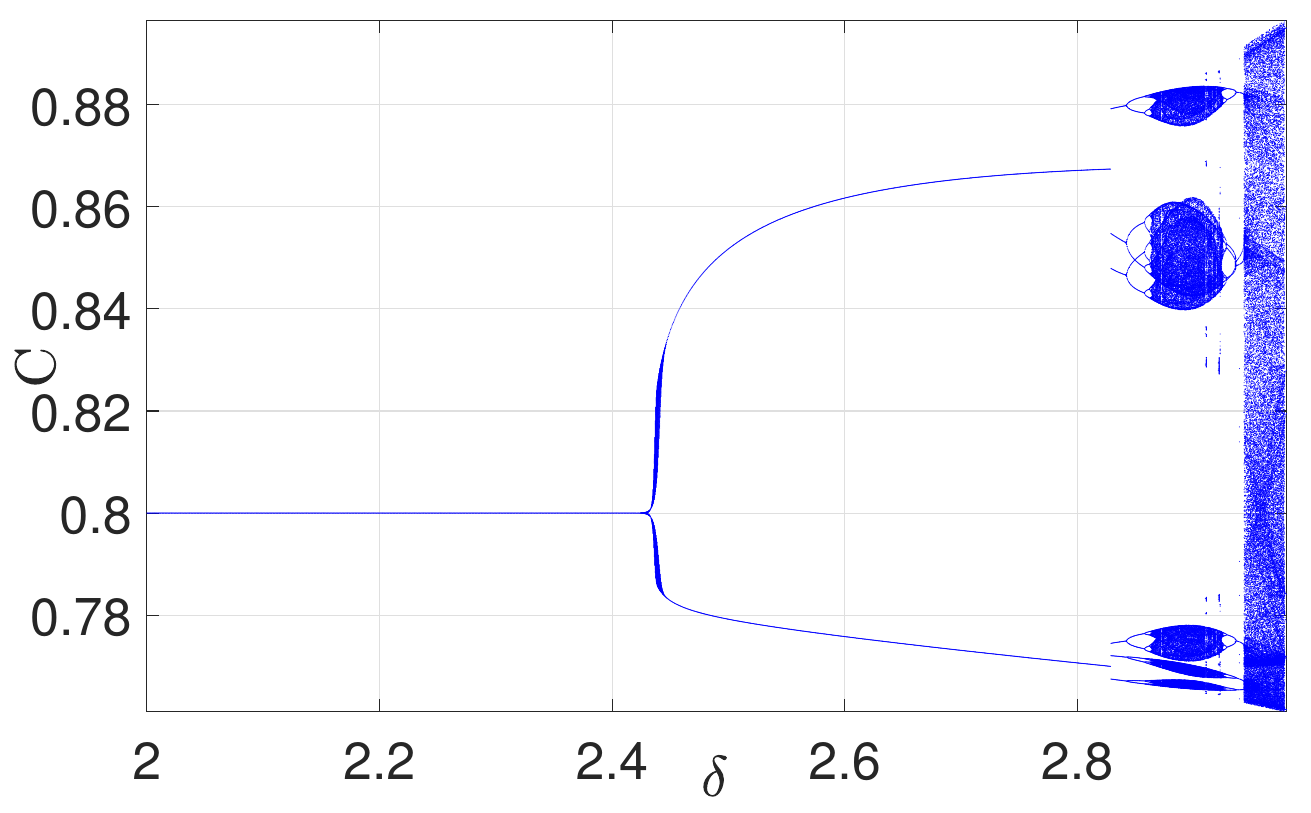}
	\end{subfigure}
	\caption{Bifurcation diagram for $r=0.798$.}
	\label{paper6-bif-r-0.798}
\end{figure}

\begin{remark}
	In summary, the observation from Figure \ref{paper6-bif-r-0.76} to Figure \ref{paper6-bif-r-0.8} underscores the significant finding that the bifurcation point of the model exhibits a delayed response when subjected to incremental increases in the intrinsic growth rate, while all other pertinent parameters remain constant.
\end{remark}

\begin{case}
	Consider $r=0.5,$ $k=0.7,$ $a=0.65,$ $g=0.3,$ $\gamma=0.4,$ $d=0.1,$ and $\alpha=0.6,$ with $\delta$ varies in the range $[2.8, 3.32].$
\end{case}

For Case 1 with the given parameters, the model experience Niemark-Sacker bifurcation with initial condition $(M_0, C_0) = (0.06, 0.7)$, $\delta$ varies in the range $[2.8, 3.32]$. The bifurcation diagram corresponds to Case 1 parameters in the $(\delta, M)$ and $(\delta, C)$ plane are given in Figure \ref{paper6-bif-1}. We fix the value of $\delta=3.1$, then the corresponding phase portrait is given in Figure \ref{paper6-bif-r-1-pp}.

\begin{case}
	Consider $r=0.6,$ $k=0.9,$ $a=0.45,$ $g=0.3,$ $\gamma=0.4,$ and $d=0.1,$ with $\delta$ varies in the range $[2.5, 3.15].$
\end{case}

\begin{figure}[h]\textit{}
	\centering
	\begin{subfigure}
		\centering
		\includegraphics[width=7.5cm,height=5cm]{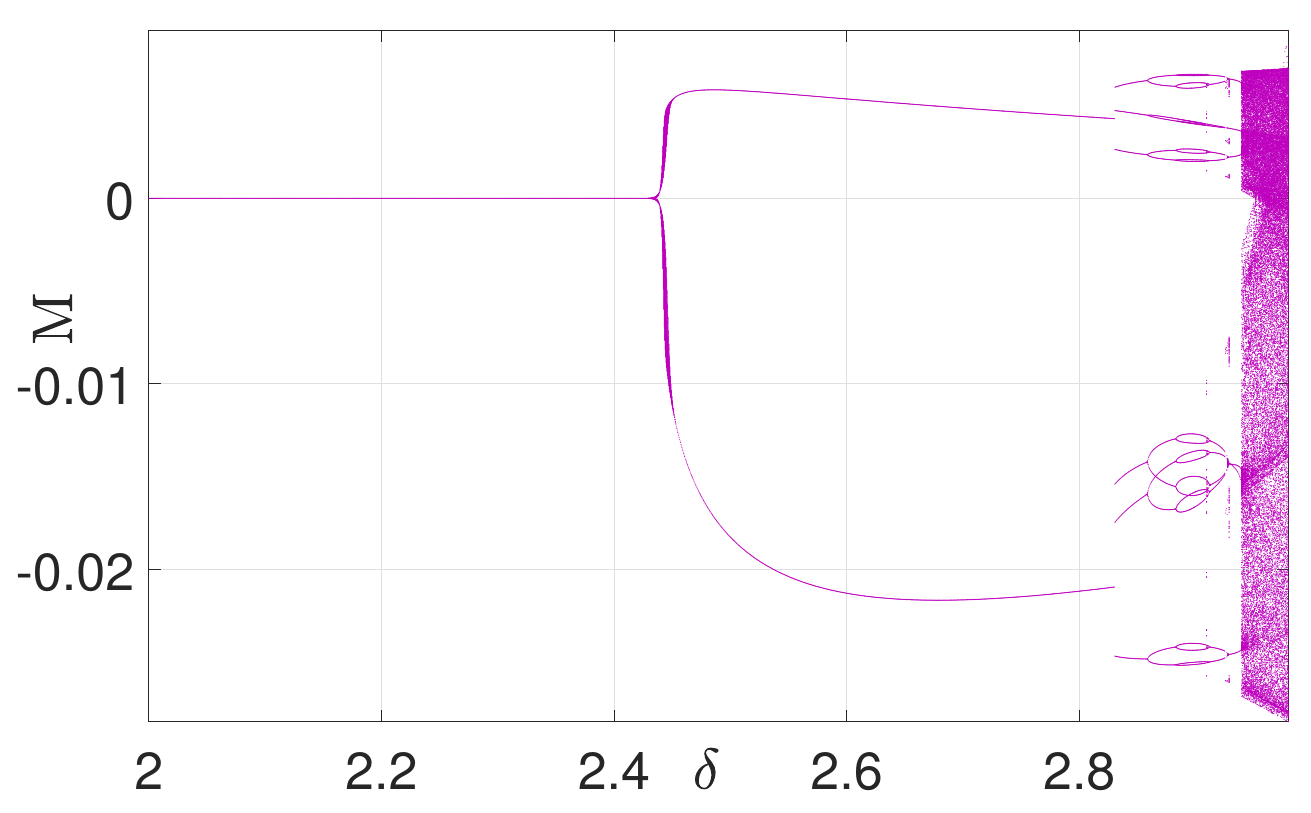}
	\end{subfigure}
	\begin{subfigure}
		\centering
		\includegraphics[width=7.5cm,height=5cm]{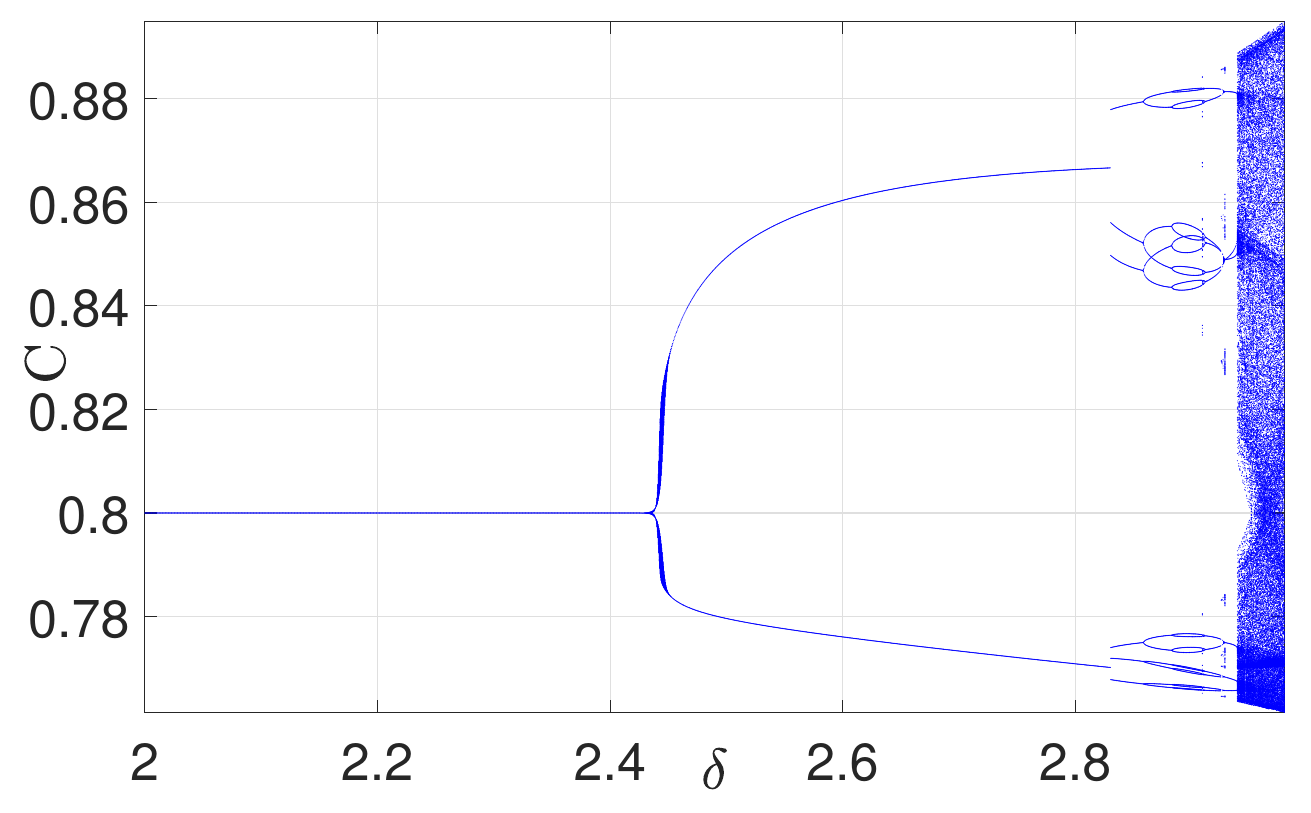}
	\end{subfigure}
	\caption{Period-24 bubble bifurcation for $\delta$ versus $M$ and $C$.}
	\label{paper6-bif-r-0.8}
\end{figure}

\begin{figure}[h]\textit{}
	\centering
	\begin{subfigure}
		\centering
		\includegraphics[width=7.5cm,height=5cm]{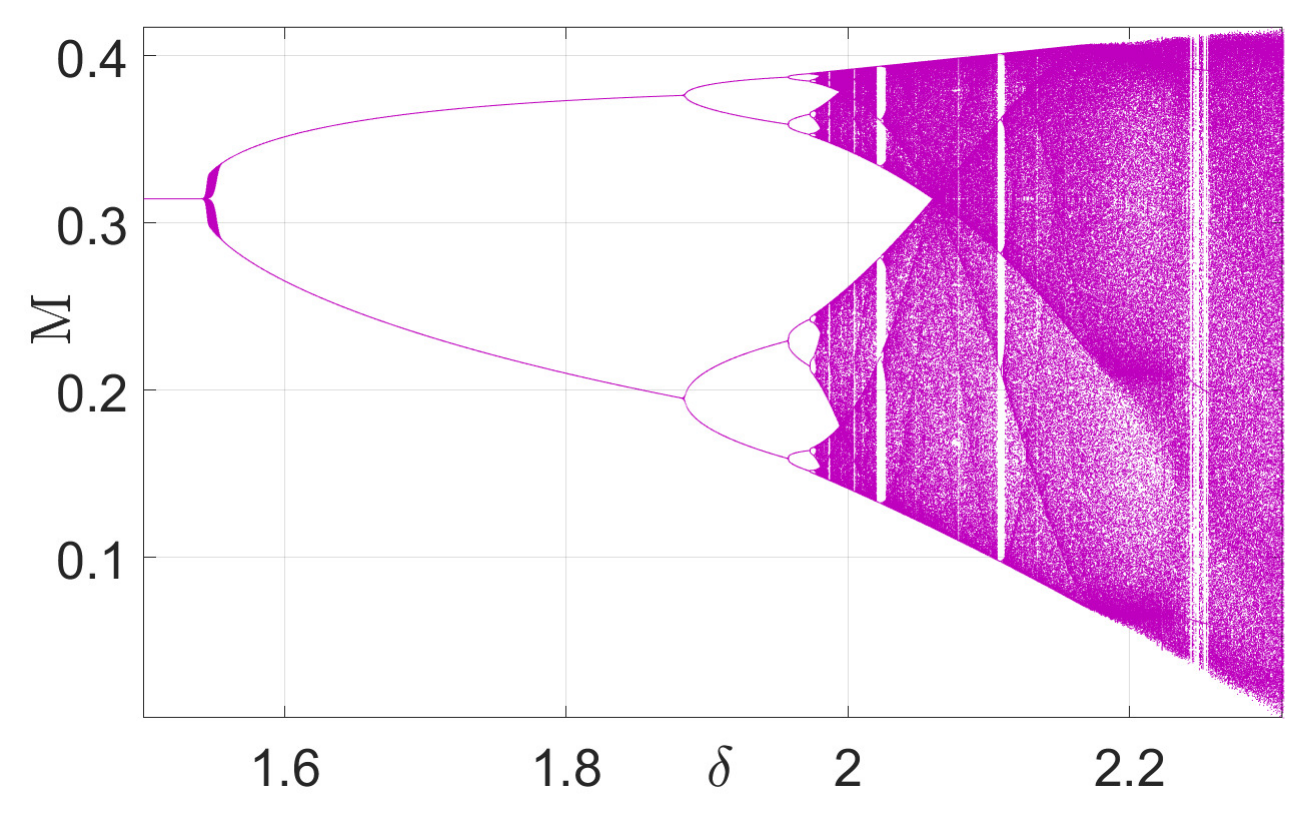}
	\end{subfigure}
	\begin{subfigure}
		\centering
		\includegraphics[width=7.5cm,height=5cm]{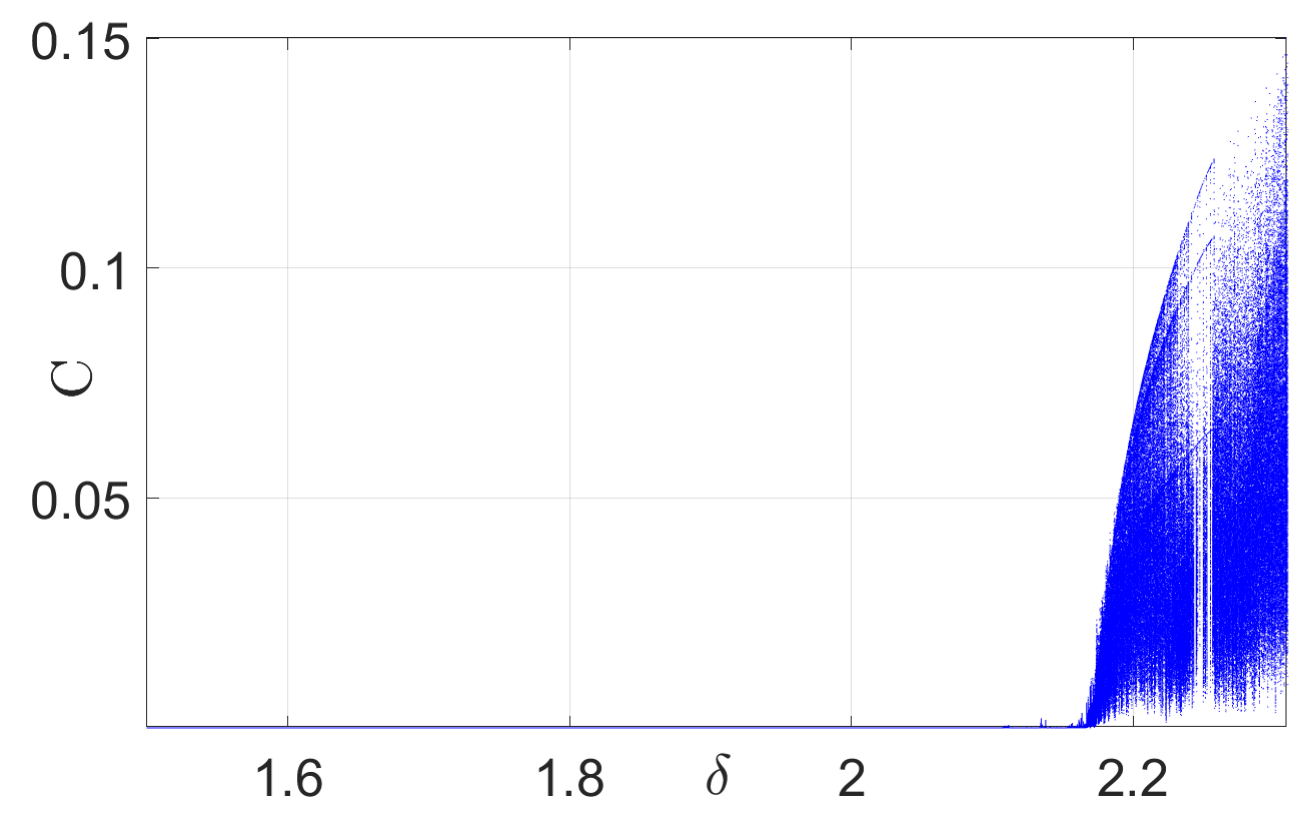}
	\end{subfigure}
	\caption{Bifurcation diagram showing the emergence of period-doubling bifurcation for $r=1.$}
	\label{paper6-bif-r-1}
\end{figure}

\begin{figure}[h]
	\begin{minipage}[b]{0.4\textwidth}
		\centering
		\includegraphics[width=\linewidth,height=3.7cm]{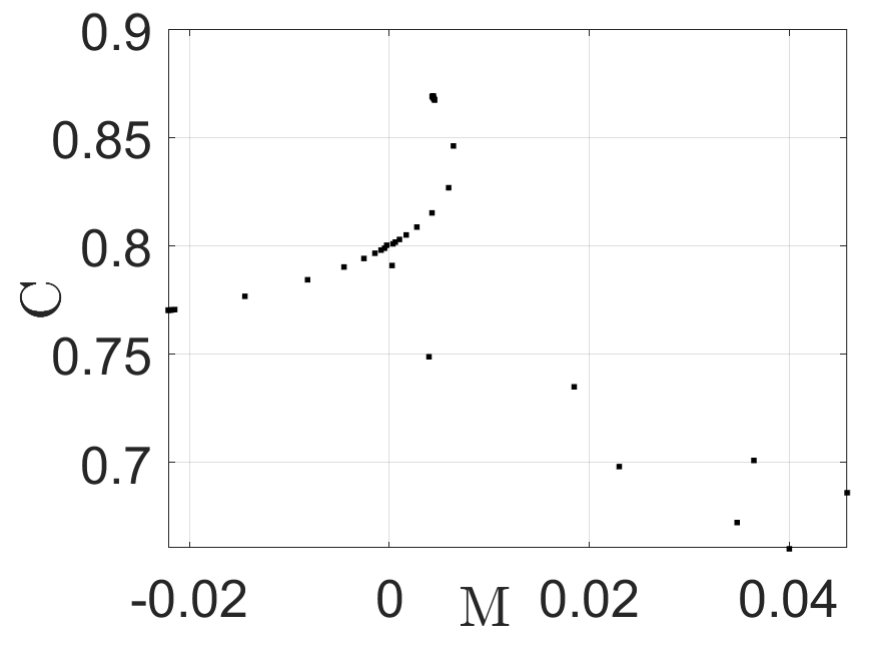}
		\caption*{(a)}
	\end{minipage}
	\begin{minipage}[b]{0.4\textwidth}
		\centering
		\includegraphics[width=\linewidth,height=3.7cm]{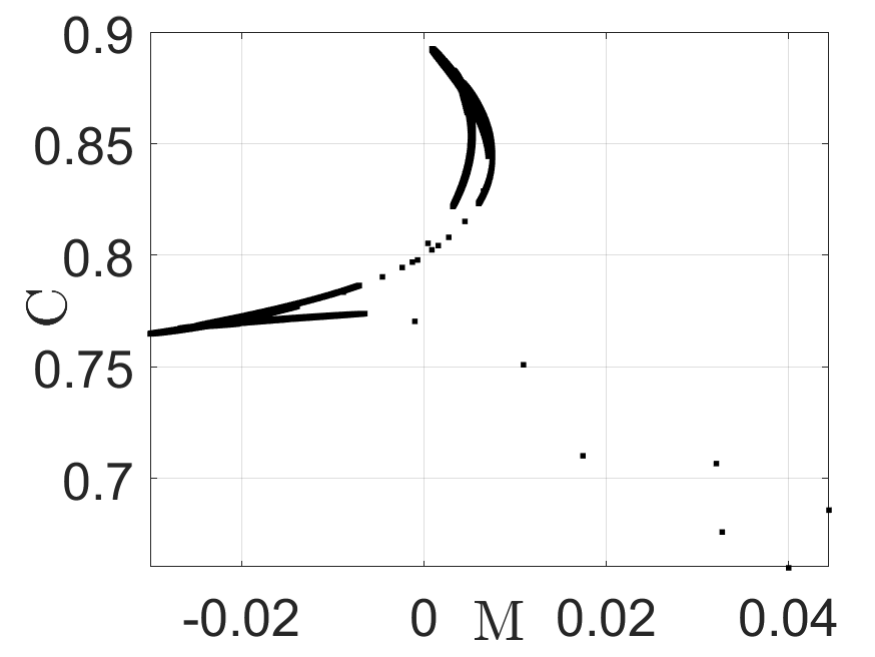}
		\caption*{(b)}
	\end{minipage}
	\begin{minipage}[b]{0.4\textwidth}
		\centering
		\includegraphics[width=\linewidth,height=3.7cm]{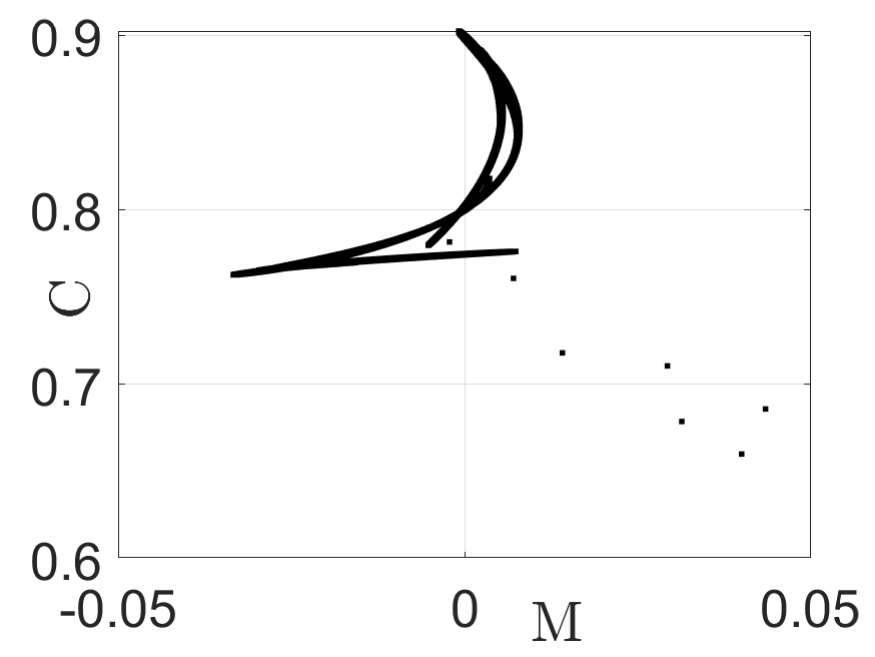}
		\caption*{(c)}
	\end{minipage}
	\begin{minipage}[b]{0.4\textwidth}
		\centering
		\includegraphics[width=\linewidth,height=3.7cm]{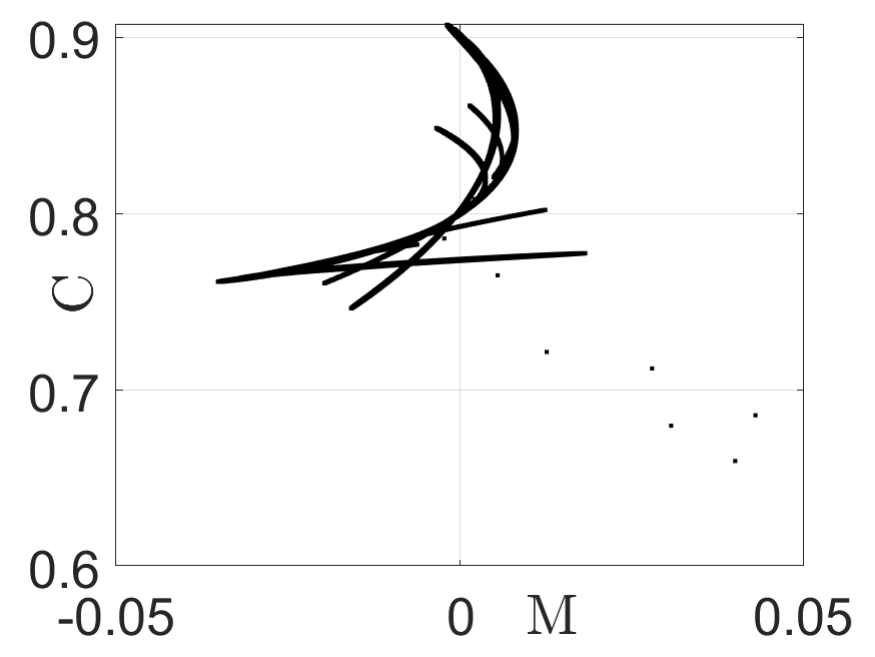}
		\caption*{(d)}
	\end{minipage}
	\begin{minipage}[b]{0.4\textwidth}
		\centering
		\includegraphics[width=\linewidth,height=3.7cm]{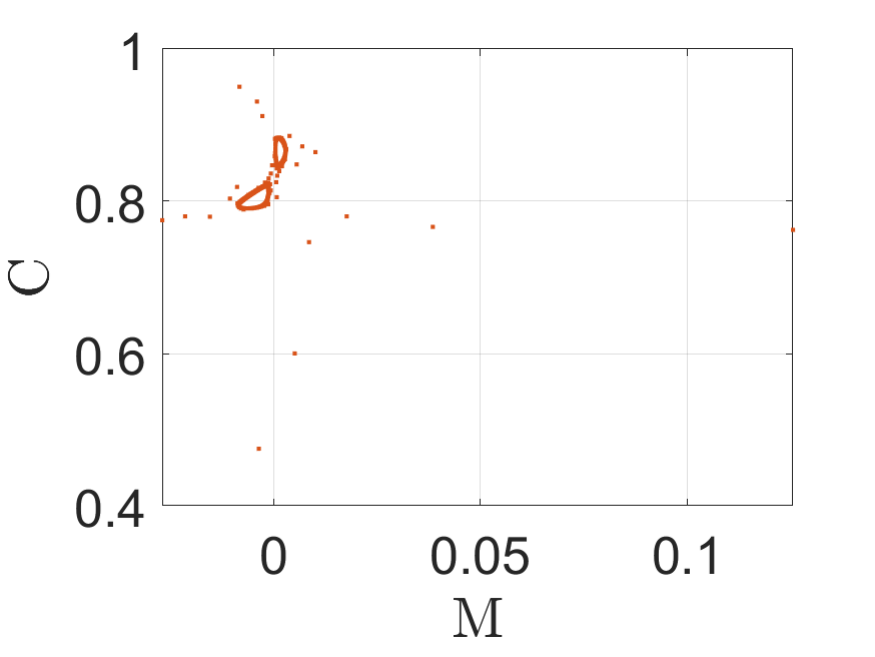}
		\caption*{(e)}
	\end{minipage}
	\begin{minipage}[b]{0.4\textwidth}
		\centering
		\includegraphics[width=\linewidth,height=3.7cm]{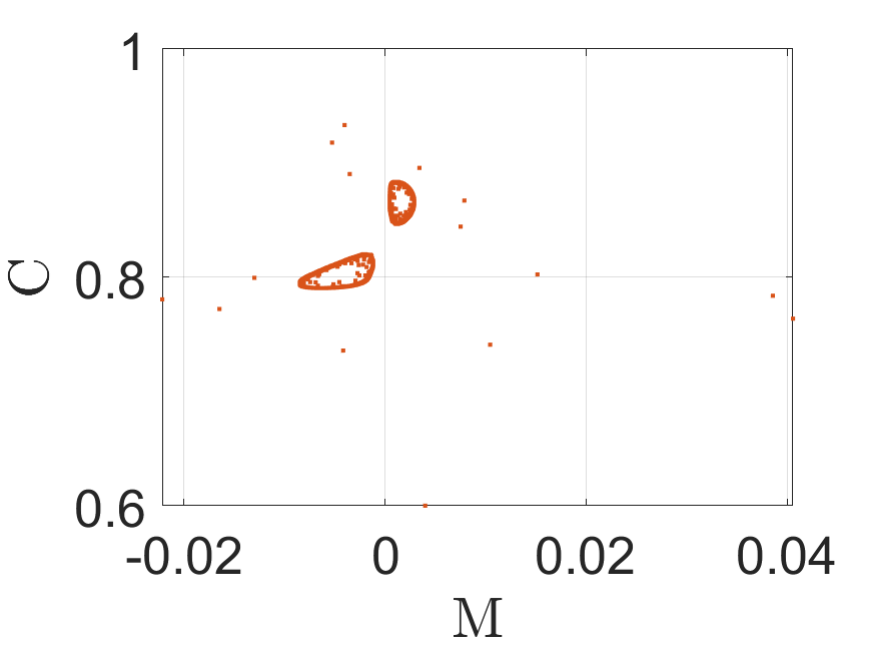}
		\caption*{(f)}
	\end{minipage}
	\begin{minipage}[b]{0.4\textwidth}
		\centering
		\includegraphics[width=\linewidth,height=3.7cm]{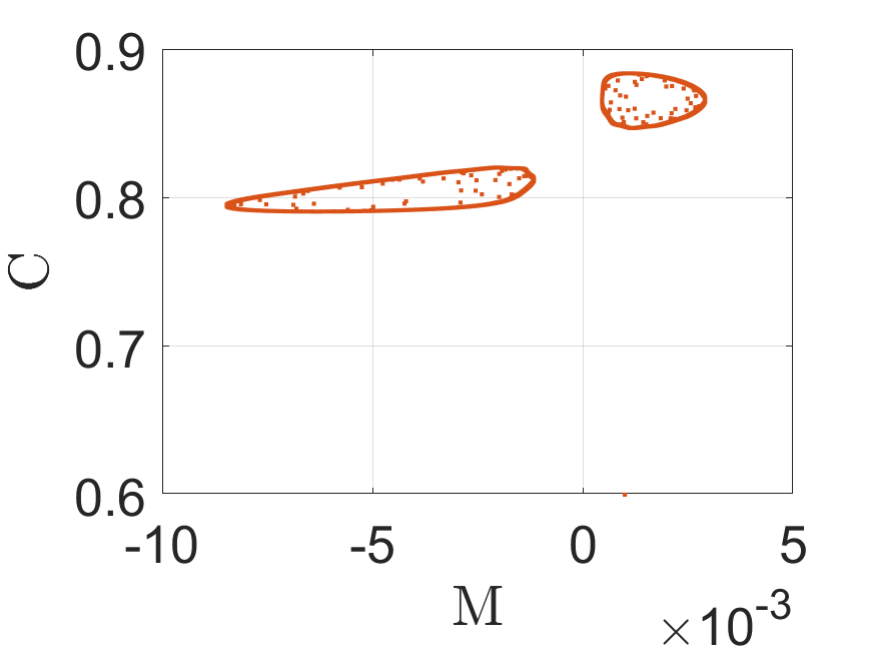}
		\caption*{(g)}
	\end{minipage}
	\caption{Phase portraits for the above set of parameters are depicted with fixed $\delta=2.8$ and varying $r$ values: (a) $r=0.795$, (b) $r=0.78$, (c) $r=0.77$, and (d) $r=0.765$. Phase portraits with initial conditions: (e) (0.005, 0.6), (f) (0.004, 0.6), and (g) (0.001, 0.6).}
	\label{Paper6-pp-0.795}
\end{figure}

For Case 2 with the given parameters, the model experience Niemark-Sacker bifurcation with initial condition $(M_0, C_0) = (0.08, 0.68)$, $\delta$ varies in the range $[2.5, 3.15]$. The bifurcation diagram corresponds to Case 2 parameters in the $(\delta, M)$ and $(\delta, C)$ plane are given in Figure \ref{paper6-bif-2}. To illustrate the bifurcation nature of the model \eqref{paper6-mod2} with the set of parameters in Case 2, we fix the value of $\delta$ in the given range such as $\delta$ =2.95, 3, 3.05, and 3.1. Then the corresponding phase portraits are given in Figures \ref{paper6-bif-2-pp-3} and \ref{paper6-bif-2-pp-3.1}.

\begin{remark}
	In this remark, we verify the nature of period-2 orbits bifurcates from $(0.04, 0.66).$ For this set of parameters $r=1,$ $k=0.3,$ $a=0.9,$ $g=0.5,$ $\gamma=0.8,$ $d=0.1,$ $\alpha=0.5$, we have eigenvalues $\lambda_1=-1$ and $\lambda_2=-12.3527$ for an interior equilibrium point equilibrium point $(0.04, 0.66)$ with $a_1=0.0134,$ $a_2=0.0666,$ $a_3=0,$ $b_1=-1.1347,$ $b_2=-4.7558,$ and $b_5=-0.1134.$ From $b_1,$ $b_2,$ and $b_5$ we have two discriminatory values $\Omega_1=-9.5115$ and $\Omega_2=2.3481$ Clearly, both the values $\Omega_1$ and $\Omega_2$ are non-zero real numbers. Hence, from Theorem \ref{paper6-Th2} the model experience flip bifurcation. Moreover, $\Omega_2>0$ which implies that period-2 orbits bifurcates from the $E^*(M^*, C^*)$ is stable.
\end{remark}

\begin{case}
	Consider $r=0.5,$ $k=0.7,$ $a=0.65,$ $g=0.3,$ $\gamma=0.4,$ $d=0.1,$ and $\alpha=0.6,$ with $\delta$ varies in the range $[2.8, 3.32].$
\end{case}

For Case 1 with the given parameters, the model experiences Niemark-Sacker bifurcation with initial condition $(M_0, C_0) = (0.06, 0.7)$, $\delta$ varies in the range $[2.8, 3.32]$. The bifurcation diagram corresponds to Case 1 parameters in the $(\delta, M)$ and $(\delta, C)$ plane are given in Figure \ref{paper6-bif-1}. We fix the value of $\delta=3.1$; the corresponding phase portrait is shown in Figure \ref{paper6-bif-r-1-pp}.

\begin{case}
	Consider $r=0.6,$ $k=0.9,$ $a=0.45,$ $g=0.3,$ $\gamma=0.4,$ $d=0.1,$ and $\alpha=0.6$ with $\delta$ varies in the range $[2.5, 3.15].$
\end{case}

For Case 2 with the given parameters, the model experiences Niemark-Sacker bifurcation with initial condition $(M_0, C_0) = (0.08, 0.68)$, $\delta$ varies in the range $[2.5, 3.15]$. The bifurcation diagram corresponds to Case 2 parameters in the $(\delta, M)$ and $(\delta, C)$ plane are given in Figure \ref{paper6-bif-2}. To illustrate the bifurcation nature of the model \eqref{paper6-mod2} with the set of parameters in Case 2, we fix the value of $\delta$ in the given range such as $\delta$ =2.95, 3, 3.05, and 3.1. Then, the corresponding phase portraits are shown in Figures \ref{paper6-bif-2-pp-3} and \ref{paper6-bif-2-pp-3.1}.

\begin{figure}[h]\textit{}
	\centering
	\begin{subfigure}
		\centering
		\includegraphics[width=7.5cm,height=5cm]{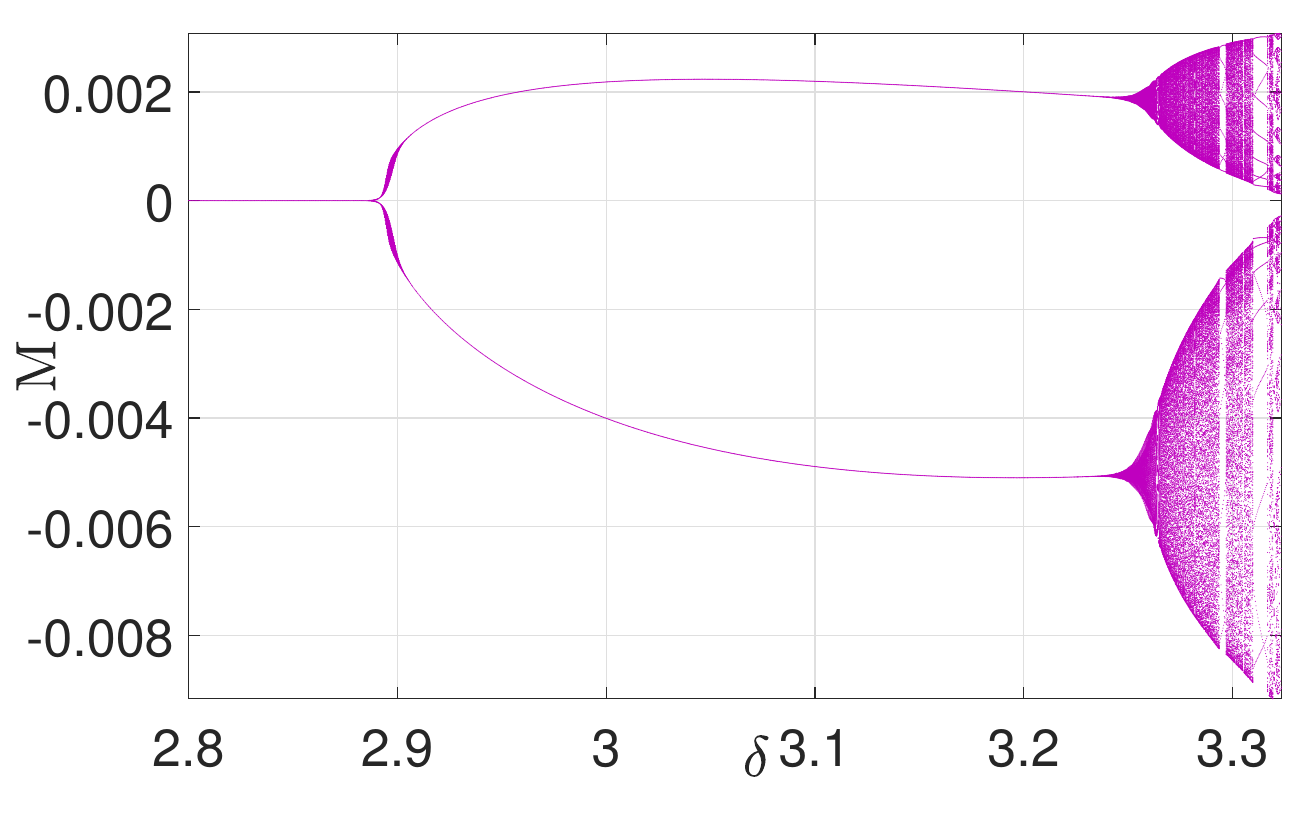}
	\end{subfigure}
	\begin{subfigure}
		\centering
		\includegraphics[width=7.5cm,height=5cm]{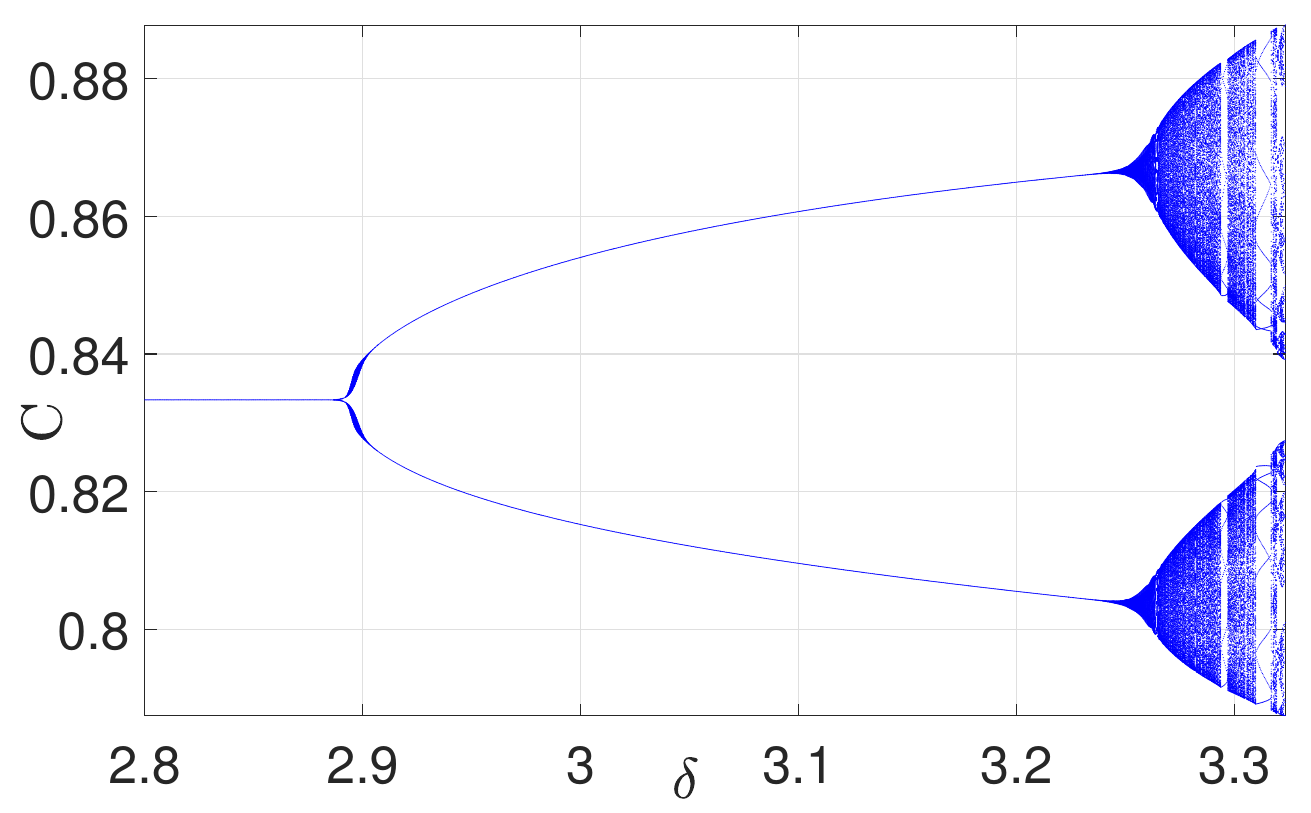}
	\end{subfigure}
	\caption{Bifurcation diagram showing the emergence of Neimark-Sacker bifurcation for Case 1.}
	\label{paper6-bif-1}
\end{figure}

\begin{figure}[h!]
	\begin{center}
		\includegraphics[width=10cm,height=5.5cm]{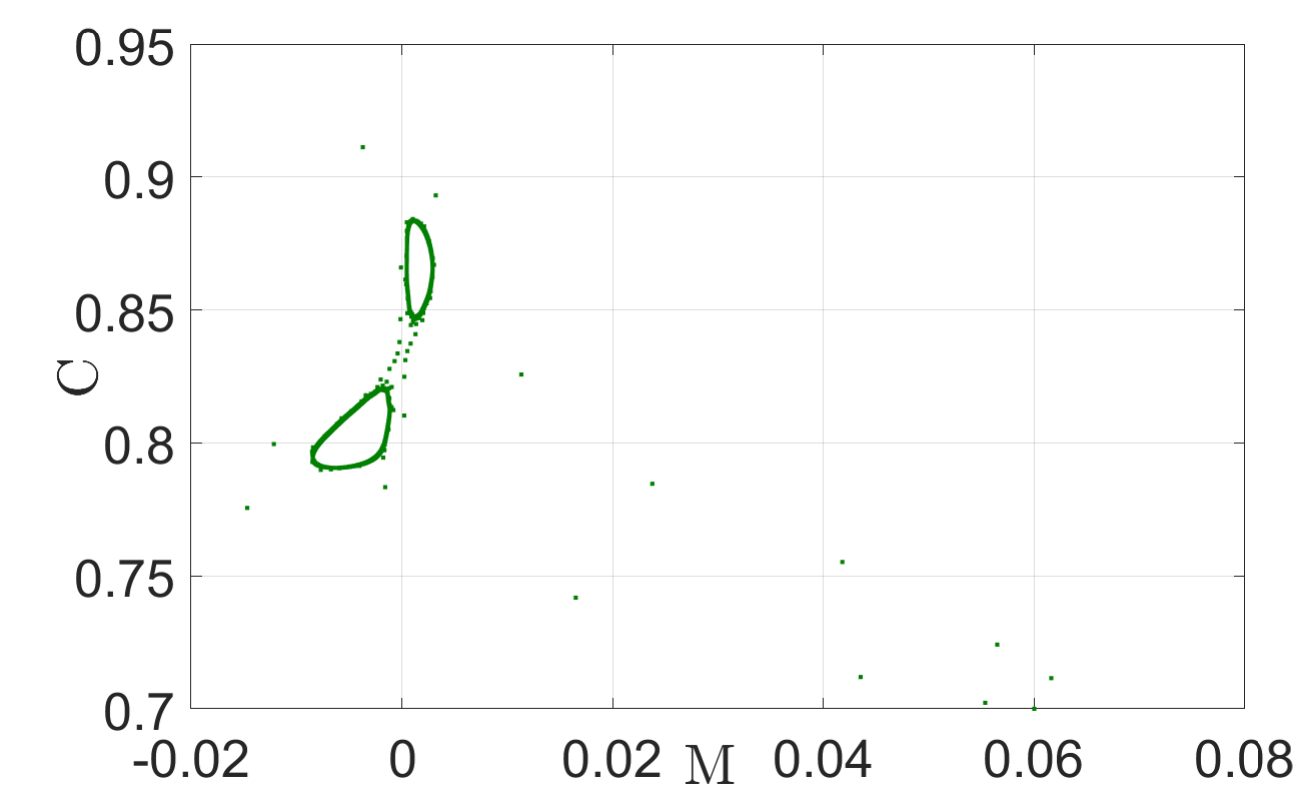}
		\caption{Phase portrait of Case 1 parameters for $\delta=3.1.$}
		\label{paper6-bif-r-1-pp}
	\end{center}
\end{figure}

\begin{figure}[h]\textit{}
	\centering
	\begin{subfigure}
		\centering
		\includegraphics[width=7.5cm,height=5cm]{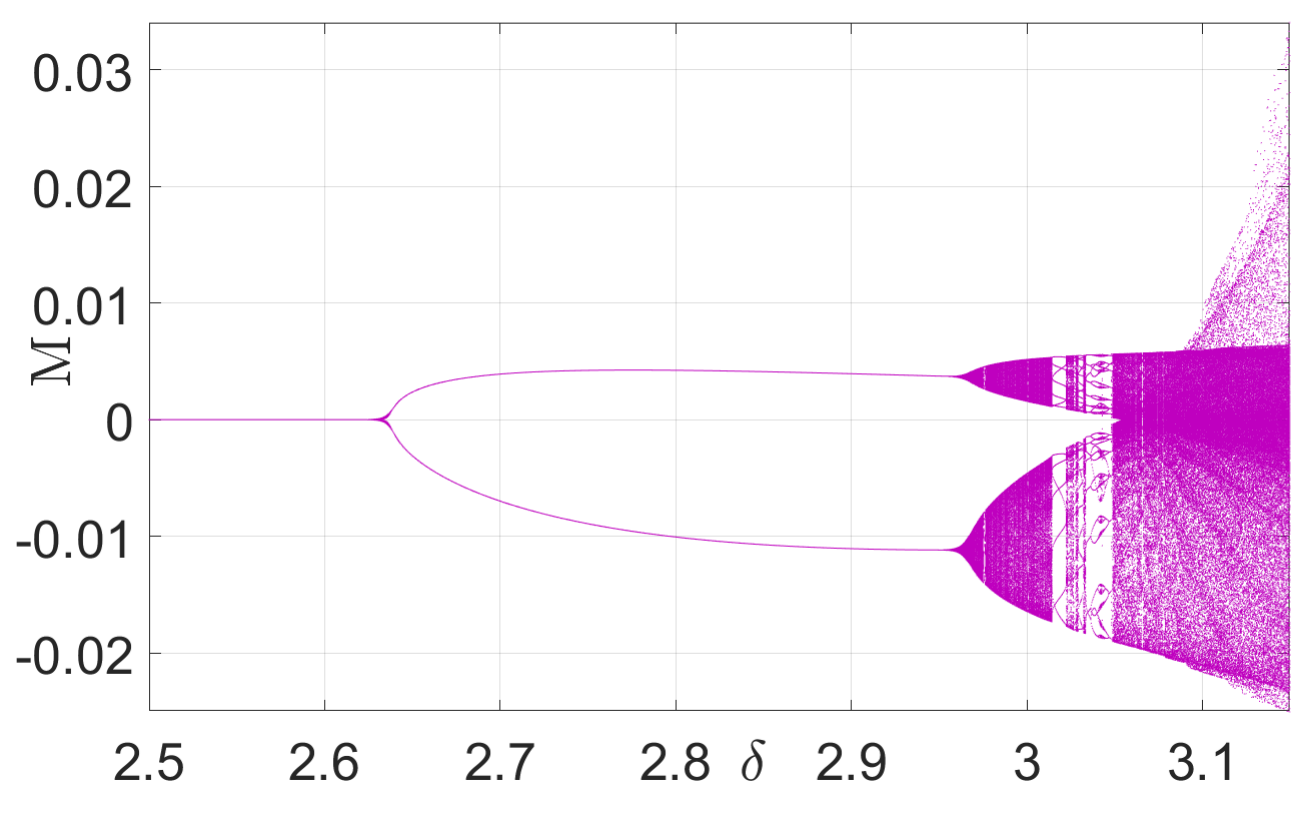}
	\end{subfigure}
	\begin{subfigure}
		\centering
		\includegraphics[width=7.5cm,height=5cm]{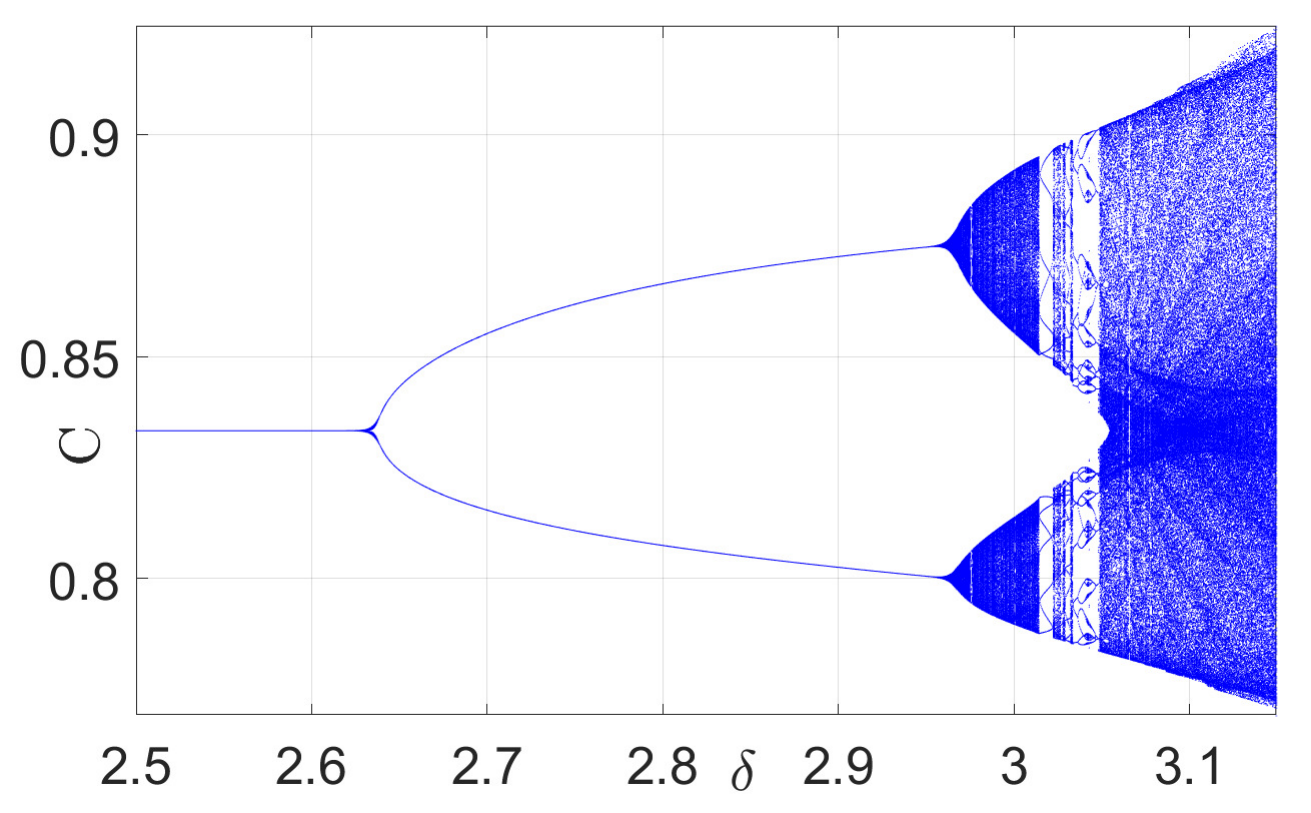}
	\end{subfigure}
	\caption{Bifurcation diagram showing the emergence of Neimark-Sacker bifurcation for Case 2.}
	\label{paper6-bif-2}
\end{figure}

\begin{remark}
	In this remark, we verify the nature of Neimark-Sacker bifurcation emergence from the fixed point $(0.06, 0.7).$ For the Case 1 parameter, we have eigenvalue 
	\begin{align*}
	\lambda,\bar{\lambda}=1+\frac{\mathcal{U}(\delta_2+\delta_2^*)}{2}\pm \frac{i (\delta_2+\delta_2^*)}{2}	\sqrt{4 \mathcal{V} - \mathcal{U}^2},	
	\end{align*}
	$\lambda=3.3318$ and $\bar{\lambda}=0.3002$ for an interior equilibrium with $\varsigma_1=4.2199+31.0638i,$ $\varsigma_2=-3.7491-16.7558i,$ $\varsigma_3=1.4140+3.2607i,$ and $\varsigma_4=-8.5369$. From $\lambda,$ $\bar{\lambda},$ $\varsigma_1,$ $\varsigma_2,$ $\varsigma_3,$ and $\varsigma_4$ we have the discriminatory value $\Psi=-125.1820 \neq 0.$ Hence, from Theorem \ref{paper6-Th3} the model experience Neimark-Sacker bifurcation. Moreover, $\Psi<0$ it indicates that repelling invariant closed curve bifurcates from $E^*(M^*, C^*).$
\end{remark}

\begin{figure}[h]\textit{}
	\centering
	\begin{subfigure}
		\centering
		\includegraphics[width=7.5cm,height=5cm]{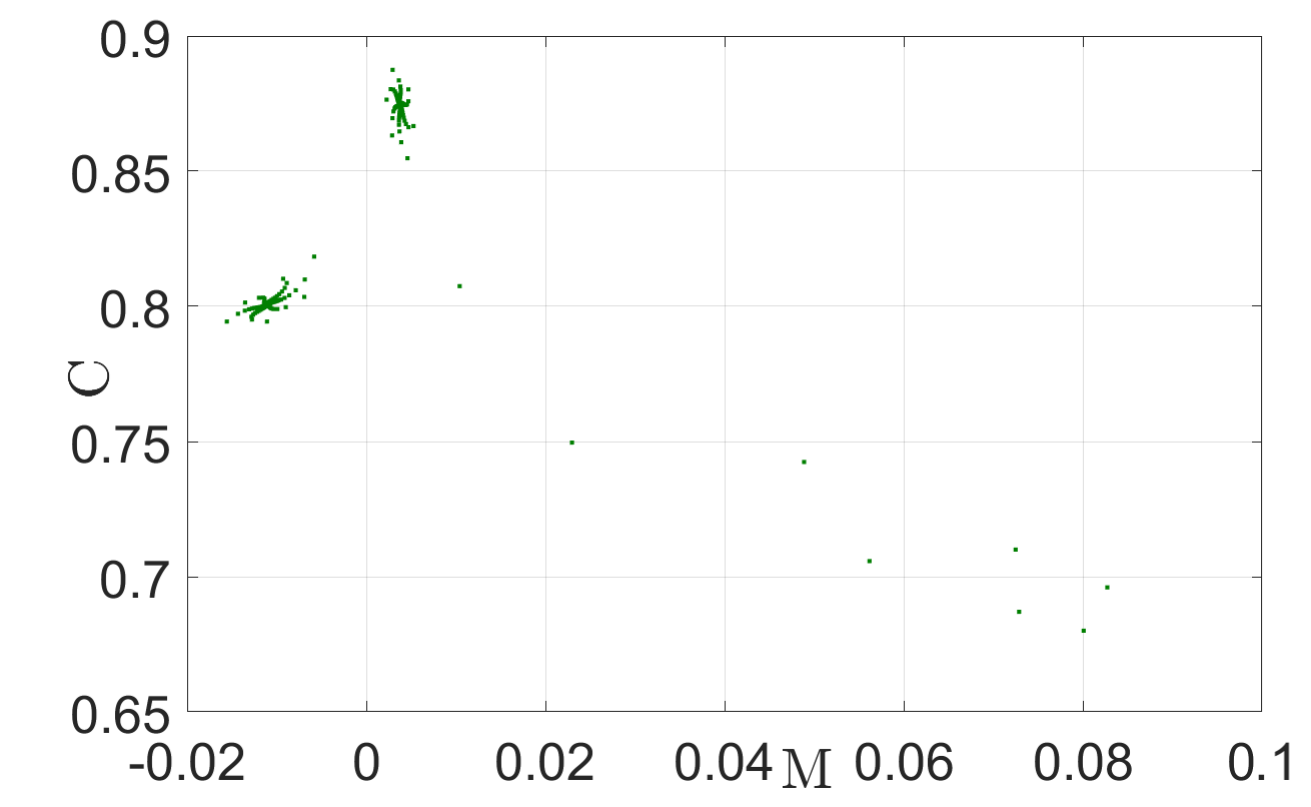}
	\end{subfigure}
	\begin{subfigure}
		\centering
		\includegraphics[width=7.5cm,height=5cm]{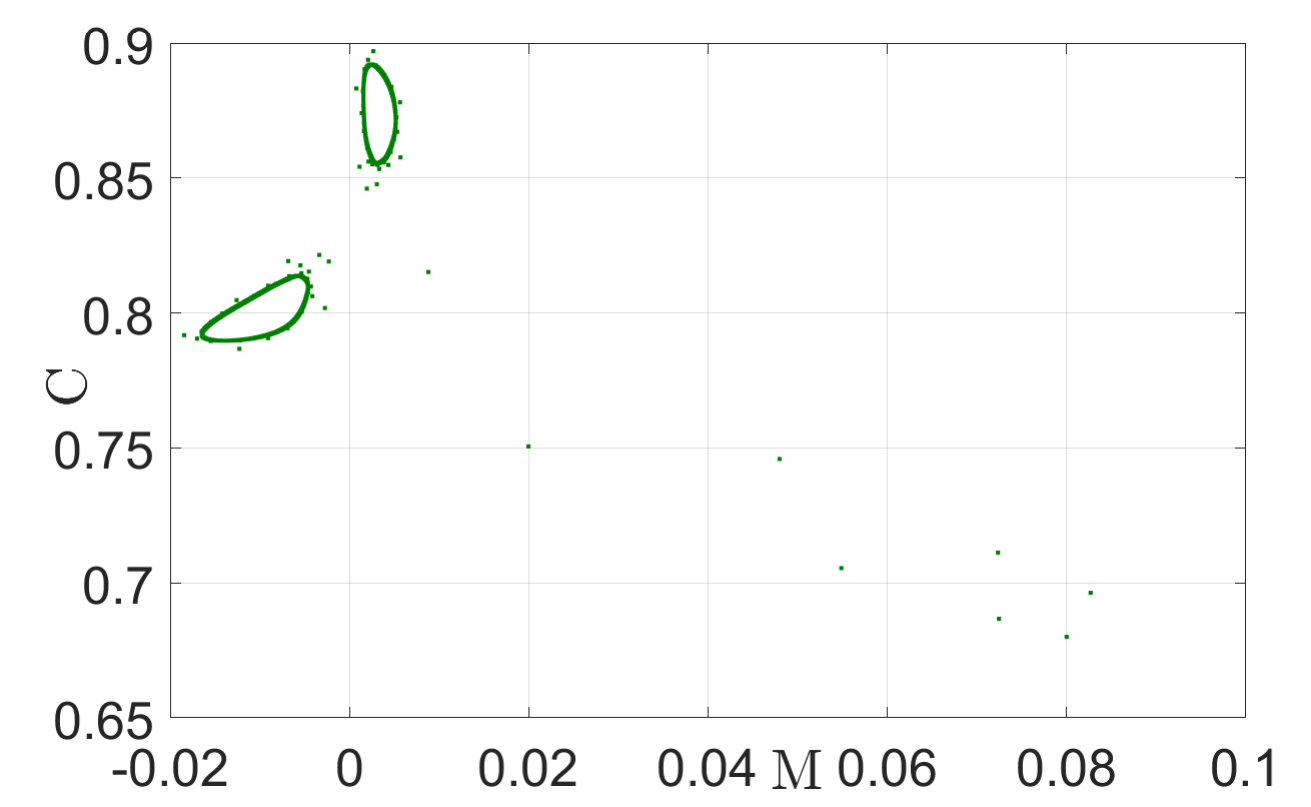}
	\end{subfigure}
	\caption{Phase portraits of Case 2 parameters for $\delta=2.95$ and $\delta=3$, respectively.}
	\label{paper6-bif-2-pp-3}
\end{figure}

\begin{figure}[h]\textit{}
	\centering
	\begin{subfigure}
		\centering
		\includegraphics[width=7.5cm,height=5cm]{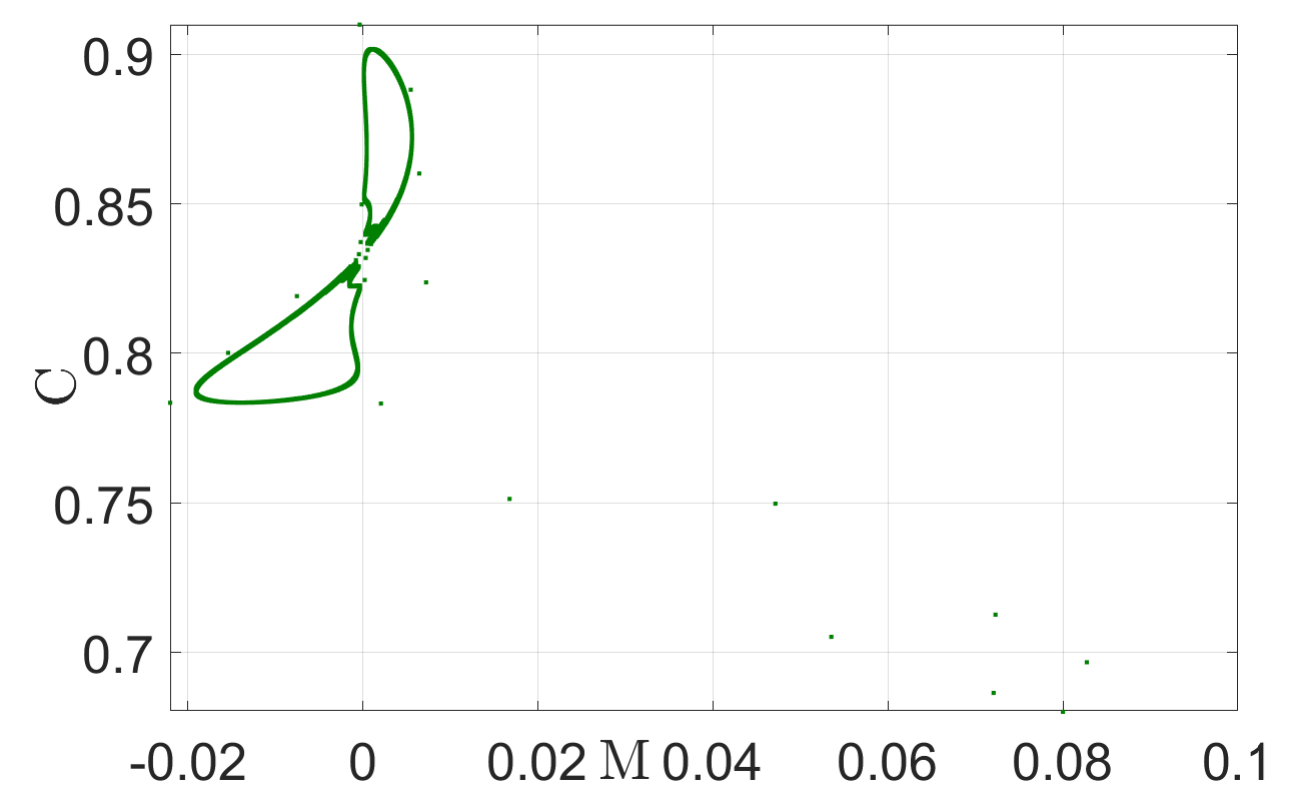}
	\end{subfigure}
	\begin{subfigure}
		\centering
		\includegraphics[width=7.5cm,height=5cm]{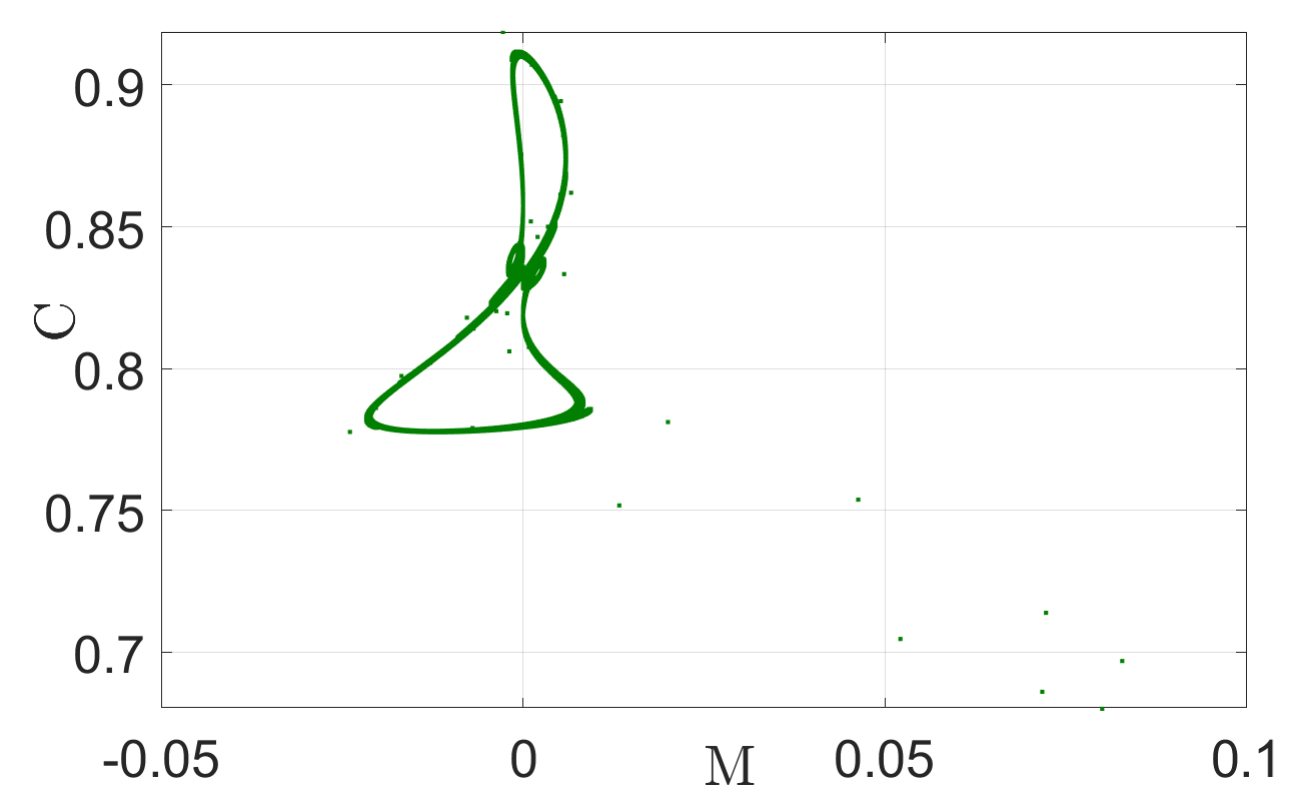}
	\end{subfigure}
	\caption{Phase portraits of Case 2 parameters for $\delta=3.05$ and $\delta=3.1$, respectively.}
	\label{paper6-bif-2-pp-3.1}
\end{figure}

\begin{remark}
	In this remark, we verify the nature of Neimark-Sacker bifurcation emergence from the fixed point $(0.08, 0.68).$ For the Case 2 parameter, we have eigenvalue $\lambda=3.3710$ and $\bar{\lambda}=0.2966$ for an interior equilibrium with $\varsigma_1=4.2725+26.3698i,$ $\varsigma_2=-3.8442-13.3660i,$ $\varsigma_3=1.5006+2.4943i,$ and $\varsigma_4=-7.0960$. From $\lambda,$ $\bar{\lambda},$ $\varsigma_1,$ $\varsigma_2,$ $\varsigma_3,$ and $\varsigma_4$ we have the discriminatory value $\Psi=-83.5551 \neq 0.$ Hence, from Theorem \ref{paper6-Th3} the model experience Neimark-Sacker bifurcation. Moreover, $\Psi<0$ it indicates that repelling invariant closed curve bifurcates from $E^*(M^*, C^*).$
\end{remark}

It is established that the model experiences Neimark-Sacker bifurcation for the given set of parameters in Case 2. We obtain the following controlled model that corresponds to the Case 2 parametric values to apply the OGY method:
\begin{align}\label{paper6-eq17}
	M_{n+1}&=M_n+\delta\left[0.6 M_n \left(1-\frac{M_n}{0.9}\right)+0.45M_nC_n-\frac{0.3 M_n}{M_n+s}+0.4M_ns\right],\nonumber\\
	C_{n+1}&=C_n+\delta\Big[0.6s -0.1 C_n-0.45 M_n C_n\Big],
\end{align}
where $\delta_0=1-\rho_1 (M_n-M^*)-\rho_2(C_n-C^*)$ and $(M^*, C^*)=(0.08, 0.28)$ is an unstable equilibrium of the model \eqref{paper6-mod11} corresponds to the given set of parameters.

Now, we have 
\begin{align*}
	J(M^*,C^*,\delta_0)=\begin{bmatrix}
	1.5050 & 0.0360\\
	-0.1260 & 1.2480
	\end{bmatrix},\quad
	B=\begin{bmatrix}
	0.0410\\
	0.0694
	\end{bmatrix},
\end{align*}
and \begin{align*}
	C=[B : JB]= \begin{bmatrix}
	0.0410 & 0.0641\\
	0.0694 & 0.0815
	\end{bmatrix}.
\end{align*}

Since, $\det(C)=-0.0011 \neq 0$, clearly $C$ is a rank 2 matrix. Hence, the model \eqref{paper6-mod2} is controllable. Consider the following Jacobian matrix of the controlled model:
\begin{align*}
	J-BH=\begin{bmatrix}
	1.5050-0.0409 \rho_1 & 0.036-0.0409 \rho_2\\
	-0.0694 \rho_1-0.126  & 1.248-0.0694\rho_2
	\end{bmatrix}.
\end{align*}

The characteristic equation of $J-BH$ is given by
\begin{align*}
	\lambda^2-\textrm{tr}(J-BH)\lambda+\det(J-BH)=0,
\end{align*}
where $\textrm{tr}(J-BH)=2.7529-0.0694 \rho_2-0.0409 \rho_1$ and $\det(J-BH)=1.8827-0.1096 \rho_2-0.0486 \rho_1.$ Moreover, the marginal stability of the controlled model \eqref{paper6-eq17} are given by
\begin{align*}
	\mathcal{L}_1=&-0.04862\rho_1-0.109665 \rho_2+0.88273,\\
	\mathcal{L}_2=&0.00765\rho_1+0.040226\rho_2-0.129767,\\
	\mathcal{L}_3=&-0.089578\rho_1-0.179106\rho_2+5.635693.
\end{align*}
	\begin{figure}[h!]
		\begin{center}
		\includegraphics[width=10cm,height=5.5cm]{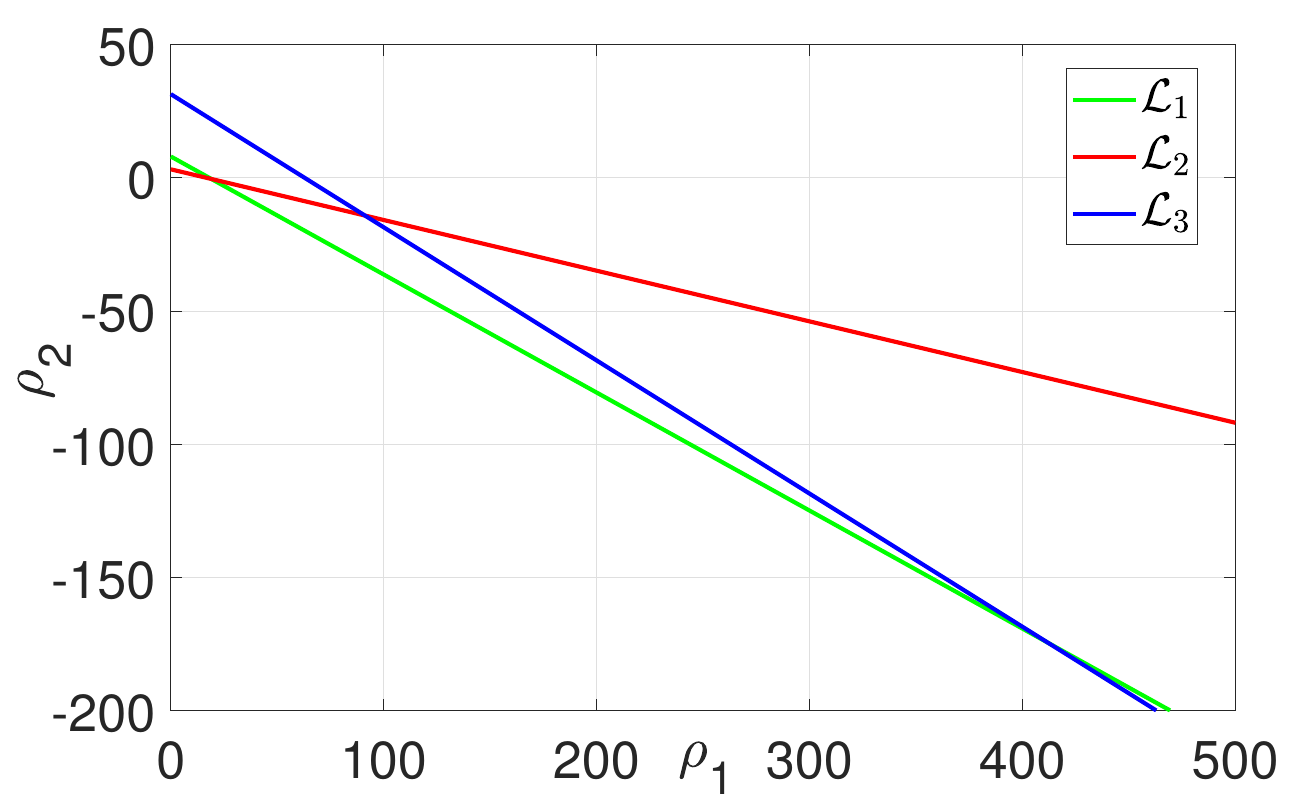}
		\caption{Stability region of the controlled model \eqref{paper6-eq17}.}
		\label{paper6-Marginal_Stability}
	\end{center}
	\end{figure}

\begin{remark}
	The triangular region covered by the straight lines $\mathcal{L}_1$, $\mathcal{L}_2$, and $\mathcal{L}_3$ in the $\rho_1\rho_2-$plane contains stable eigenvalues of the controlled model by taking the integral step size $\delta$ as the control parameter with a given parametric value. The stability region bounded by marginal lines $\mathcal{L}_1$, $\mathcal{L}_2,$ and $\mathcal{L}_3$ of model \eqref{paper6-eq17} is given in Figure \ref{paper6-Marginal_Stability}.
\end{remark}

\begin{remark}
	We have referenced the work of \cite{paper6-Li}, who examined the stability and Hopf-bifurcation analysis of the coral reef ODE model. Our study focuses on a slightly modified version of this model proposed by \cite{paper6-Li}. Our analysis has identified codimension-one period-doubling bifurcation and various complex dynamic behaviors, including reversal period-doubling bifurcation, period-4, 8, and 24 bubble bifurcations, and chaotic behavior. These findings highlight the richer dynamics exhibited by discrete-time models compared to continuous-time models, contributing to our understanding of chaotic behaviors in ecological systems.
\end{remark}

\section*{Conclusion}
This paper examines the complex dynamics of the discrete-time coral reef dynamical model incorporating macroalgae. We demonstrate that the model undergoes codimension-one bifurcation Neimark-Sacker bifurcation. Furthermore, the proposed model exhibits even more exciting dynamic behavior, including invariant cycles and chaotic sets. These results illustrate that the discrete model has far richer dynamics than the continuous model. We employ feedback control to stabilize the chaotic orbits at an unstable fixed point emerging through Neimark-Sacker bifurcation. Finally, we verify the theoretical results numerically with the help of suitable computational software. In the future, the authors will focus on the bifurcation analysis of a four-dimensional coral reef mathematical model that incorporates parrotfish interaction.

\section*{Statements and Declarations}
\noindent \textbf{Animal Research:} This article does not contain any studies with human participants or animals performed by any of the authors.\\
\noindent \textbf{Author Contribution:} M. Priyanka: Conceptualization, Methodology, Software, Writing - original draft, Writing - review \& editing. P. Muthukumar: Conceptualization, Validation, Supervision, Formal analysis, Writing - review \& editing.\\
\noindent \textbf{Code Availability:} Not applicable.\\
\noindent \textbf{Competing Interests:} The authors have no competing interests to declare that are relevant to the content of this article.\\
\noindent \textbf{Conflict of Interest:} The authors declare that they have no conflict of interest.\\
\noindent \textbf{Data Availability Statement:}
Data sharing does not apply to this article as no datasets were generated or analyzed during the current study.\\
\noindent \textbf{Financial/Non-financial Interests:} The authors have no relevant financial or non-financial interests to disclose.

\textbf{Appendix A}

\textbf{Proof of Theorem \ref{paper6-Th2}}

Let $x=M-M^*,$ $y=C-C^*$. Then shift the interior equilibrium point $E^*(M^*, C^*)$ of map \eqref{paper6-eq2} into the origin which given by
\begin{align}\label{paper6-eq3}
	\left\{ \begin{array}{l}
		x \longrightarrow i_{100}x+i_{010}y+i_{200}x^2+i_{110}xy+i_{300}x^3+i_{101}x\delta_1^*+i_{011}y\delta_1^*+i_{201}x^2\delta_1^*\\
		\qquad\quad\!\!+i_{111}xy\delta_1^*+O\Big((|x|+|y|+|\delta_1^*|)^4\Big),\\
		y \longrightarrow j_{100}x+j_{010}y+j_{110}xy+j_{101}x\delta_1^*+j_{011}y\delta_1^*+j_{111}xy\delta_1^*+O\Big((|x|+|y|+|\delta_1^*|)^4\Big),
	\end{array}
	\right.
\end{align}
where
$i_{100}=1+\delta r-\frac{2 \delta r M}{k}+\delta a C-\frac{\delta S g}{(M+S)^2}+\delta \gamma S,$ $i_{010}=\delta aM,$ $i_{200}=\frac{-\delta r}{k}+\frac{\delta S g}{(M+S)^3},$ $i_{110}=\delta a,$ $i_{101}=r-\frac{2rM}{k}+aC-\frac{Sg}{(M+S)^2}+\gamma S,$ $i_{300}=\frac{-\delta S g}{(M+S)^4},$ $i_{011}=aM,$ $i_{201}=\frac{-r}{k}+\frac{S g}{(M+S)^3},$ $i_{111}=a,$ $j_{100}=-\delta a C,$ $j_{010}=1+\delta \alpha S-\delta d-\delta aM,$ $j_{110}=-\delta a,$ $j_{101}=-aC,$ $j_{011}=\alpha S-d-aM,$ $j_{111}=-a,$ and $\delta=\delta_1$. Assume $i_{010} \neq 0.$

The invertible matrix is constructed as,
\begin{align*}
	P_1=\begin{bmatrix}
		i_{010} & i_{010}\\
		-1-i_{100} & \lambda_2-i_{100}
	\end{bmatrix}
\end{align*}
and by using this translation
\begin{align}\label{paper6-eq18}
	\begin{bmatrix}
		x\\
		y
	\end{bmatrix}=P_1\begin{bmatrix}
		\tilde{M}\\
		\tilde{C}
	\end{bmatrix},
\end{align}
which implies
\begin{align*}
	\begin{bmatrix}
		\tilde{M}\\
		\tilde{C}
	\end{bmatrix}=\frac{1}{|P_1|}\begin{bmatrix}
		\lambda_2 - i_{100} & -i_{010}\\
		1+i_{100} & i_{010}
	\end{bmatrix}\begin{bmatrix}
		x\\
		y
	\end{bmatrix},
\end{align*}
where $|P_1|=i_{010}(1+\lambda_2)$. Then

\begin{align*}
	\begin{bmatrix}
		\tilde{M}\\
		\tilde{C}
	\end{bmatrix}=	\frac{1}{i_{010}(1+\lambda_2)}\begin{bmatrix}
		(\lambda_2-i_{100})x-i_{010}y\\
		(1+i_{100})x+i_{010}y
	\end{bmatrix}.
\end{align*}
By map \eqref{paper6-mod2}, we have
\begin{align*}
	\begin{bmatrix}
		\tilde{M}\\
		\tilde{C}
	\end{bmatrix}\longrightarrow \frac{1}{i_{010}(1+\lambda_2)} \begin{bmatrix}
		(\lambda_2-i_{100}) \Big[ i_{100}x+i_{010}y+i_{200}x^2+i_{110}xy+i_{300}x^3+i_{101}x\delta_1^*+i_{011}y\delta_1^*&\\
		+i_{201}x^2\delta_1^*+i_{111}xy\delta_1^*+O\Big((|x|+|y|+|\delta_1^*|)^4\Big) \Big]&\\
		- i_{010}\Big[ j_{100}x+j_{010}y+j_{110}xy+j_{101}x\delta_1^*+j_{011}y\delta_1^*&\\
		\vspace{0.5cm}
		+j_{111}xy\delta_1^*+O\Big((|x|+|y|+|\delta_1^*|)^4\Big) \Big]&\\
		(1+i_{100})\Big[ i_{100}x+i_{010}y+i_{200}x^2+i_{110}xy+i_{300}x^3+i_{101}x\delta_1^*+i_{011}y\delta_1^*&\\
		+i_{201}x^2\delta_1^*+i_{111}xy\delta_1^*+O\Big((|x|+|y|+|\delta_1^*|)^4\Big) \Big]&\\
		+i_{010} \Big[ j_{100}x+j_{010}y+j_{110}xy+j_{101}x\delta_1^*+j_{011}y\delta_1^*&\\
		+j_{111}xy\delta_1^*+O\Big((|x|+|y|+|\delta_1^*|)^4\Big) \Big]&
	\end{bmatrix}.
\end{align*}
Now, the above map \eqref{paper6-eq3} can be rewritten as 
\begin{align}\label{paper6-eq4}
	\begin{bmatrix}
		\tilde{M}\\
		\tilde{C}
	\end{bmatrix} \longrightarrow \begin{bmatrix}
		-1 & 0\\
		0 & \lambda_2
	\end{bmatrix} \begin{bmatrix}
		\tilde{M}\\
		\tilde{C}
	\end{bmatrix}+\begin{bmatrix}
		g_1(\tilde{M}, \tilde{C}, \delta_1^*)\\
		g_2(\tilde{M}, \tilde{C}, \delta_1^*)
	\end{bmatrix},
\end{align}

where 
\begin{align*}
	g_1(\tilde{M}, \tilde{C}, \delta_1^*)=\frac{1}{i_{010}(1+\lambda_2)} \begin{bmatrix}
		\Big[(\lambda_2-i_{100})i_{110}-i_{010}j_{110} \Big]xy+\Big[(\lambda_2-i_{100})i_{101}-i_{010}j_{101} \Big]x\delta_1^*&\\
		+\Big[(\lambda_2-i_{100})i_{011}-i_{010}j_{011} \Big]y\delta_1^*+\Big[(\lambda_2-i_{100})i_{111}-i_{010}j_{111} \Big]xy\delta_1^*&\\
		+(\lambda_2-i_{100})i_{200}x^2+(\lambda_2-i_{100})i_{201}x^2\delta_1^*&\\
		+(\lambda_2 -i_{100})i_{300}x^3+O\Big((|x|+|y|+|\delta_1^*|)^4\Big)&
	\end{bmatrix},
\end{align*}

\begin{align*}
	g_2(\tilde{M}, \tilde{C}, \delta_1^*)=\frac{1}{i_{010}(1+\lambda_2)} \begin{bmatrix}
		\Big[(1+i_{100})i_{110}+i_{010}j_{110} \Big]xy+\Big[(1+i_{100})i_{101}+i_{010}j_{101} \Big]x\delta_1^*&\\
		+\Big[(1+i_{100})i_{011}+i_{010}j_{011} \Big]y\delta_1^*+\Big[(1+i_{100})i_{111}+i_{010}j_{111} \Big]xy\delta_1^*&\\
		+(1+i_{100})i_{200}x^2+(1+i_{100})i_{201}x^2\delta_1^*&\\
		+(1+i_{100})i_{300}x^3+O\Big((|x|+|y|+|\delta_1^*|)^4\Big)&
	\end{bmatrix},
\end{align*}
and from \eqref{paper6-eq18}, we have
\begin{align*}
	\begin{bmatrix}
		x\\
		y
	\end{bmatrix}=\begin{bmatrix}
		i_{010}\tilde{M}+i_{010}\tilde{C}\\
		-(1+i_{100})\tilde{M}+(\lambda_2-i_{100})\tilde{C}
	\end{bmatrix},
\end{align*}
which implies
\begin{align*}
	x&=i_{010}(\tilde{M}+\tilde{C}),\\
	y&=-(1+i_{100})\tilde{M}+(\lambda_2-i_{100})\tilde{C},\\
	xy&=i_{010}\Big[-(1+i_{100})\tilde{M}^2+(\lambda_2-i_{100})\tilde{C}^2+(\lambda_2-1-2i_{100})\tilde{M}\tilde{C}\Big],\\
	x^2&=i_{010}^2(\tilde{M}^2+\tilde{C}^2+2\tilde{M}\tilde{C}),\\
	x^3&=i_{010}^3(\tilde{M}^3+\tilde{C}^3+3\tilde{M}^2\tilde{C}+3\tilde{M}\tilde{C}^2).
\end{align*}
Next, the center manifold $C_M(0,0,0)$ of \eqref{paper6-eq4} about the origin in a small neighborhood of $\delta_1^*$ is determined. We can conclude that there is a center manifold $C_M(0,0,0)$ from the center manifold theorem given as follows:
\begin{align*}
	C_M(0,0,0)=\left\{(\tilde{M}, \tilde{C}, \delta_1^*)\in \mathbb{R}^3, \tilde{C}=\mathfrak{F}(\tilde{M}, \delta_1^*), \mathfrak{F}(0,0)=0, D\mathfrak{F}(0,0)=0 \right\},
\end{align*}
for $\tilde{M}$ and $\delta_1^*$ sufficiently small. Assume the following center manifold form:
\begin{align}\label{paper6-eq15}
	\mathfrak{F}(\tilde{M}, \delta_1^*)=a_1 \tilde{M}^2+a_2\tilde{M} \delta_1^*+a_3{\delta_1^*}^2+O\Big((|\tilde{M}|+|\delta_1^*|)^3\Big).
\end{align} 
Given $C_M(0,0,0)$ must satisfy  
\begin{align}\label{paper6-eq16}
	\mathfrak{F}\Big(-\tilde{M}+g_1(\tilde{M}, \mathfrak{F}(\tilde{M},\delta_1^*),\delta_1^*),\delta_1^*\Big)=\lambda_2\mathfrak{F}(\tilde{M},\delta_1^*)+g_2\Big(\tilde{M}, \mathfrak{F}(\tilde{M},\delta_1^*),\delta_1^*\Big).
\end{align}

Substituting \eqref{paper6-eq15} into \eqref{paper6-eq16}, and comparing coefficients of similar powers in obtained \eqref{paper6-eq16}, we get that
\begin{align*}
	a_1=&\frac{1}{i_{010}(1-\lambda_2^2)}\Big[-[(1+i_{100})i_{110}+i_{010}j_{110} ]i_{010}(1+i_{100})+(1+i_{100})i_{200}i_{010}^2 \Big],\\
	a_2=&\frac{-1}{i_{010}(1+\lambda_2)^2}\Big[[ (1+i_{100})i_{101}+i_{010}j_{101}]i_{010}-[(1+i_{100})i_{011}+i_{010}j_{011}  ](1+i_{100})\Big],\\
	a_3=&0.
\end{align*}

The map \eqref{paper6-eq4} that is restricted to the central manifold $C_M(0,0,0)$ is taken into consideration:
\begin{align}\label{paper6-eq5}
	G: \tilde{M} \longrightarrow -\tilde{M}+b_1\tilde{M}^2+b_2\tilde{M} \delta_1^*+b_3\tilde{M}^2\delta_1^*+b_4\tilde{M}{\delta_1^*}^2+b_5\tilde{M}^3+O\Big((|\tilde{M}|+|\delta_1^*|)^4\Big),
\end{align}
where 
\begin{align*}
	b_1=\frac{1}{i_{010}(1+\lambda_2)}&\Big[-[(\lambda_2-i_{100})i_{110}-i_{010}j_{110} ]i_{010}(1+i_{100})+(\lambda_2-i_{100})i_{200}i_{010}^2 \Big],\\
	b_2=\frac{1}{i_{010}(1+\lambda_2)}&\Big[[(1+i_{100})i_{101}+i_{010}j_{101}]i_{010}-[(1+i_{100})i_{011}+i_{010}j_{011}](1+i_{100}) \Big],\\
	b_3=\frac{1}{i_{010}(1+\lambda_2)}&\Big[[(\lambda_2-i_{100})i_{110}-i_{010}j_{110} ]i_{010}(\lambda_2-1-2i_{100})a_2+[(\lambda_2-i_{100})i_{101}-i_{010}j_{101} ]i_{010}a_1\\ &+[(\lambda_2-i_{100})i_{011}-i_{010}j_{011}](\lambda_2-i_{100})a_1-[(\lambda_2-i_{100})i_{111}-i_{010}j_{111} ]i_{010}(1+i_{100}) \\ &+2(\lambda_2-i_{100})i_{200}i_{010}^2a_2+(\lambda_2-i_{100})i_{201}i_{010}^2\Big],\\
	b_4=\frac{1}{i_{010}(1+\lambda_2)}&\Big[[(\lambda_2-i_{100})i_{101}-i_{010}j_{101} ]a_2i_{010}+[(\lambda_2-i_{100})i_{011}-i_{010}j_{011} ](\lambda_2-i_{100})a_2 \Big],\\
	b_5=\frac{1}{i_{010}(1+\lambda_2)}&\Big[[(\lambda_2-i_{100})i_{110}-i_{010}j_{110} ]i_{010}(\lambda_2-1-2i_{100} )a_1+2(\lambda_2-i_{100})i_{200}i_{010}^2a_1\\
	&+(\lambda_2 - i_{100}) i_{300} i_{010}^3 \Big].
\end{align*}

Two discriminatory values $\Omega_1$ and $\Omega_2$, must not be zero for the map \eqref{paper6-eq5} to experience a flip bifurcation.
\begin{align*}
	\Omega_1&=\left.\left(2\frac{\partial^2G}{\partial \tilde{M} \partial \delta_1^*}+\frac{\partial G}{\partial \delta_1^*} \frac{\partial G}{\partial \tilde{M}} \right)\right|_{(0,0)}=2b_2,\\
	\Omega_2&=\left.\left(\frac{1}{2}\left(\frac{\partial^2 G}{\partial \tilde{M}^2} \right)^2+\frac{1}{3}\left(\frac{\partial^3 G}{\partial \tilde{M}^3} \right) \right)\right|_{(0,0)}=2b_1^2+2b_5.
\end{align*}

\textbf{Appendix B}

\textbf{Proof of Theorem \ref{paper6-Th3}}

Let $x=M-M^*$, $y=C-C^*$. Then shift $E^*(M^*, C^*)$ of map \eqref{paper6-eq2} into the origin which is given by
\begin{align}\label{paper6-eq7}
	\begin{bmatrix}
		x\\
		y
	\end{bmatrix}\longrightarrow \begin{bmatrix}
		i_{100}x+i_{010}y+i_{200}x^2+i_{020}y^2+i_{110}xy+i_{120}xy^2+i_{030}y^3+O\Big((|x|+|y|)^4\Big)\\
		j_{100}x+j_{010}y+j_{020}y^2+j_{110}xy+O\Big((|x|+|y|)^4\Big)
	\end{bmatrix},
\end{align}
where $i_{100}$, $i_{010}$, $i_{200},$ $i_{020},$ $i_{110}$, $i_{120},$ $i_{300},$ $j_{100}$, $j_{010}$, $j_{020},$ $j_{110}$, are given previously by substituting $\delta$ for $\delta_2+\delta_2^*.$

Consider the following characteristic equation for the linearization of the map \eqref{paper6-eq7} about the origin:
\begin{align*}
	\lambda^2+\mathcal{A}(\delta_2^*)\lambda+\mathcal{B}(\delta_2^*)=0,	
\end{align*}
where
\begin{align*}
	\mathcal{A}(\delta_2^*)&=-2-\mathcal{U}(\delta_2+\delta_2^*),\\
	\mathcal{B}(\delta_2^*)&=1+\mathcal{U}(\delta_2+\delta_2^*)+\mathcal{V}(\delta_2+\delta_2^*)^2.
\end{align*}

Given that the parameters are  $(r, k, a, g, \gamma, \alpha, d, \delta_2) \in \mathcal{N}$, the eigenvalues of $(0, 0)$ are a pair of complex conjugate numbers called $\lambda$ and $\bar{\lambda}$ with modulus 1 by Theorem \ref{paper6-Th1}, where
\begin{align*}
	\lambda, \bar{\lambda}=\frac{-\mathcal{A}(\delta_2^*)\pm i\sqrt{4\mathcal{B}(\delta_2^*)-\mathcal{A}^2(\delta_2^*)}}{2}
\end{align*}
and 
\begin{align*}
	|\lambda|_{\delta_2^*=0}=\sqrt{\mathcal{B}(0)}=1, \quad \left. \frac{d |\lambda|}{d\delta_2^*} \right|_{\delta_2^*=0}=-\frac{\mathcal{U}}{2}>0.
\end{align*}

Additionally, it is required that when $\delta_2^*=0$, $\lambda^r,$ $\bar{\lambda}^r \neq 1$ $(r=1,2,3,4)$ which is equivalent to $\mathcal{A}(0) \neq -2, 0, 1, 2$. Note that $(r, k, a, g, \gamma, \alpha, d, \delta_2) \in \mathcal{N}.$ Thus, $\mathcal{A}(0) \neq -2, 2$, we only need to show that $\mathcal{A}(0) \neq 0,1,$ which implies
\begin{align}\label{paper6-eq14}
	\mathcal{U}^2 \neq 2\mathcal{V}, 3\mathcal{V}.
\end{align}

When $\delta_2^* = 0$ and \eqref{paper6-eq14} is true, the eigenvalues of fixed point $(0, 0)$ of \eqref{paper6-eq7} do not lie in the position where the coordinate axes and unit circle coincide. Then, we investigate the normal form of \eqref{paper6-eq7} at $\delta_2^*=0$.

Let $\sigma_1=1+\frac{\mathcal{U} \delta_2}{2}$ and $\sigma_2=\frac{\delta_2}{2}\sqrt{4\mathcal{V}-\mathcal{U}^2}$. The invertible matrix is constructed as,
\begin{align*}
	P_2=\begin{bmatrix}
		i_{010} & 0\\
		\sigma_1-i_{100} & -\sigma_2
	\end{bmatrix}
\end{align*}
and by using this translation
\begin{align}\label{paper6-eq8}
	\begin{bmatrix}
		x\\
		y
	\end{bmatrix}=P_2\begin{bmatrix}
		\tilde{M}\\
		\tilde{C}
	\end{bmatrix},
\end{align}
which implies
\begin{align*}
	\begin{bmatrix}
		\tilde{M}\\
		\tilde{C}
	\end{bmatrix}=\frac{1}{|P_2|}\begin{bmatrix}
		-\sigma_2 & 0\\
		i_{100}-\sigma_1 &i_{010}
	\end{bmatrix}\begin{bmatrix}
		x\\
		y
	\end{bmatrix},
\end{align*}
where $|P_2|=-\sigma_2 i_{010}$. Now,
\begin{align*}
	\begin{bmatrix}
		\tilde{M}\\
		\tilde{C}
	\end{bmatrix}=\frac{1}{-\sigma_2i_{010}}\begin{bmatrix}
		-\sigma_2x\\
		(i_{100}-\sigma_1)x+i_{010}y
	\end{bmatrix}.
\end{align*}
By map \eqref{paper6-mod2}, we have
\begin{align*}
	\begin{bmatrix}
		\tilde{M}\\
		\tilde{C}
	\end{bmatrix}\longrightarrow \frac{1}{-\sigma_2i_{010}}\begin{bmatrix}
		-\sigma_2\Big[i_{100}x+i_{010}y+i_{200}x^2+i_{020}y^2+i_{110}xy+i_{120}xy^2\\		\vspace{0.5cm}
		+i_{030}y^3+O\Big((|x|+|y|)^4\Big)\Big]\\
		(i_{100}-\sigma_1)\Big[i_{100}x+i_{010}y+i_{200}x^2+i_{020}y^2+i_{110}xy+i_{120}xy^2\\
		+i_{030}y^3+O\Big((|x|+|y|)^4\Big)\Big]\\
		+i_{010}\Big[j_{100}x+j_{010}y+j_{020}+j_{110}xy+O\Big((|x|+|y|)^4\Big)\Big]
	\end{bmatrix}.
\end{align*}
Now, the above matrix can be rewritten as 
\begin{align}\label{paper6-eq19}
	\begin{bmatrix}
		\tilde{M}\\
		\tilde{C}
	\end{bmatrix}\!\longrightarrow \! \frac{1}{-\sigma_2i_{010}}\!\begin{bmatrix}
		-\sigma_2i_{100} & -\sigma_2i_{010}\\
		(i_{100}-\sigma_1)i_{100}+i_{010}j_{100} & (i_{100}-\sigma_1)i_{010}+i_{010}j_{010}
	\end{bmatrix}
	\begin{bmatrix}
		x\\
		y
	\end{bmatrix}\!+\!\begin{bmatrix}
		g_3(\tilde{M}, \tilde{C})\\
		g_4(\tilde{M}, \tilde{C})
	\end{bmatrix},
\end{align}
where 
\begin{align*}
	g_3(\tilde{M}, \tilde{C})=\frac{1}{-\sigma_2i_{010}}&\Bigg[-\sigma_2\Big[i_{200}x^2++i_{020}y^2+i_{110}xy+i_{120}xy^2+i_{030}y^3+O\Big((|x|+|y|)^4\Big) \Big]\Bigg],\\
	g_4(\tilde{M},\tilde{C})=\frac{1}{-\sigma_2i_{010}}&\Bigg[(i_{100}-\sigma_1)\Big[i_{200}x^2++i_{020}y^2+i_{110}xy+i_{120}xy^2+i_{030}y^3+O\Big((|x|+|y|)^4\Big)\Big]\\
	&+i_{010}\Big[j_{020}y^2+j_{110}xy+O\Big((|x|+|y|)^4\Big)\Big]\Bigg],
\end{align*}
and from \eqref{paper6-eq8}, we have
\begin{align*}
	\begin{bmatrix}
		x\\
		y
	\end{bmatrix}=\begin{bmatrix}
		i_{010}\tilde{M}\\
		(\sigma_2-i_{100})\tilde{M}-\sigma_2\tilde{C}
	\end{bmatrix},
\end{align*}
which implies,
\begin{align*}
	x&=i_{010}\tilde{M},\\
	y&=(\sigma_1-i_{100})\tilde{M}-\sigma_2\tilde{C},\\
	xy&=i_{010}(\sigma_1-i_{100})\tilde{M}^2-\sigma_2i_{010}\tilde{M}\tilde{C},\\
	x^2&=i_{010}^2\tilde{M}^2,\\
	y^2&=(\sigma_1-i_{100})^2\tilde{M}^2+\sigma_2^2\tilde{C}^2-2\sigma_2(\sigma_1-i_{100})\tilde{M}\tilde{C},\\
	xy^2&=i_{010}(\sigma_1-i_{100})^2\tilde{M}^3+i_{010}\sigma_2^2\tilde{M}\tilde{C}^2-2\sigma_2i_{010}(\sigma_1-i_{100})\tilde{M}^2\tilde{C},\\
	y^3&=(\sigma_1-i_{100})^3\tilde{M}^3-\sigma_2^3\tilde{C}^3+3\sigma_2^2(\sigma_1-i_{100})\tilde{M}\tilde{C}^2-3\sigma_2(\sigma_1-i_{100})^2\tilde{M}^2\tilde{C},
\end{align*}
and
\begin{align*}
	&g_{3\tilde{M}\tilde{M}}=\frac{-1}{\sigma_2i_{010}}\Big[-\sigma_2[2i_{200}i_{010}^2+2i_{020}(\sigma_1-i_{100})^2+2i_{110}i_{010}(\sigma_1-i_{100})+i_{120}[6i_{010}(\sigma_1-i_{100})^2\tilde{M}\\
	&\qquad\qquad\qquad\quad-4\sigma_2i_{010}(\sigma_1-i_{100})\tilde{C}]+i_{030}[6(\sigma_1-i_{100})^3\tilde{M}-6\sigma_2(\sigma_1-i_{100})^2\tilde{C}]] \Big],\\
	&g_{3\tilde{C}\tilde{C}}=\frac{-1}{\sigma_2i_{010}}\Big[-\sigma_2[2 i_{020}\sigma_2^2+2i_{120}i_{010}\sigma_2^2\tilde{M}+i_{030}[-6\sigma_2^3\tilde{C}+6\sigma_2^2(\sigma_1-i_{100})\tilde{M}]] \Big],\\
	&g_{3\tilde{M}\tilde{C}}=\frac{-1}{\sigma_2i_{010}}\Big[-\sigma_2[-2i_{020}\sigma_2(\sigma_1-i_{100})-i_{110}i_{010}\sigma_2+i_{120}[2i_{010}\sigma_2^2\tilde{C}-4\sigma_2i_{010}(\sigma_1-i_{100})\tilde{M}]\\
	&\qquad\qquad\qquad\quad+i_{030}[6\sigma_2^2(\sigma_1-i_{100})\tilde{C}-6\sigma_2(\sigma_1-i_{100})^2\tilde{M}]] \Big],\\
	&g_{3\tilde{M}\tilde{M}\tilde{M}}=\frac{-1}{\sigma_2i_{010}} \Big[-\sigma_2 [6i_{120}i_{010}(\sigma_1-i_{100})^2+6i_{030}(\sigma_1-i_{100})^3 ] \Big],\\
	&g_{3\tilde{M}\tilde{C}\tilde{C}}=\frac{-1}{\sigma_2i_{010}} \Big[-\sigma_2 [2i_{120}i_{010}\sigma_2^2+6i_{030}\sigma_2^2(\sigma_1-i_{100}) ] \Big],\\
	&g_{3\tilde{M}\tilde{M}\tilde{C}}=\frac{-1}{\sigma_2i_{010}} \Big[-\sigma_2 [-4i_{120}\sigma_2i_{010}(\sigma_1-i_{100})-6i_{030}\sigma_2(\sigma_1-i_{100})^2 ] \Big],\\
	&g_{3\tilde{C}\tilde{C}\tilde{C}}=\frac{-1}{\sigma_2i_{010}} \Big[-\sigma_2 [-6i_{030}\sigma_2^3] \Big],\\
	&g_{4\tilde{M}\tilde{M}}=\frac{-1}{\sigma_2i_{010}}\Big[(i_{100}-\sigma_1)[2i_{010}^2i_{200}+2i_{020}(\sigma_1-i_{100})^2+2i_{110}i_{010}(\sigma_1-i_{100})\\
	&\qquad\qquad\qquad\quad+i_{120}[6i_{010}(\sigma_1-i_{100})^2\tilde{M}-4\sigma_2i_{010}(\sigma_1-i_{100})\tilde{C}]\\
	&\qquad\qquad\qquad\quad+i_{030}[6(\sigma_1-i_{100})^3\tilde{M}-6\sigma_2(\sigma_1-i_{100})^2\tilde{C}]]\\
	&\qquad\qquad\qquad\quad+i_{010}[2j_{020}(\sigma_1-i_{100})^2+2j_{110}i_{010}(\sigma_1-i_{100})]\Big],\\
	&g_{4\tilde{C}\tilde{C  }}=\frac{-1}{\sigma_2i_{010}}\Big[(i_{100}-\sigma_1)[2i_{020}\sigma_2^2+2i_{120}i_{010}\sigma_2^2\tilde{M}-6i_{030}\sigma_2^3\tilde{C}+6i_{030}\sigma_2^2(\sigma_1-i_{100})\tilde{M}]\\
	&\qquad\qquad\qquad\quad+i_{010}[2j_{020}\sigma_2^2] \Big],\\
	&g_{4\tilde{M}\tilde{C}}=\frac{-1}{\sigma_2i_{010}}\Big[(i_{100}-\sigma_1)[-2i_{020}\sigma_2(\sigma_1-i_{100})-i_{110}i_{010}\sigma_2+2i_{120}i_{010}\sigma_2^2\tilde{C}\\
	&\qquad\qquad\qquad\quad-4i_{120}\sigma_2i_{010}(\sigma_1-i_{100})\tilde{M}+6i_{030}\sigma_2^2(\sigma_1-i_{100})\tilde{C}-6i_{030}\sigma_2(\sigma_1-i_{100})^2\tilde{M}]\\
	&\qquad\qquad\qquad\quad+i_{010}[-2j_{020}\sigma_2(\sigma_1-i_{100})-j_{110}i_{010}\sigma_2] \Big],\\
	&g_{4\tilde{M}\tilde{M}\tilde{M}}=\frac{-1}{\sigma_2i_{010}}\Big[(i_{100}-\sigma_1)[6i_{120}i_{010}(\sigma_1-i_{100})^2+6 i_{030}(\sigma_1-i_{100})^3 ] \Big],\\
	&g_{4\tilde{C}\tilde{C}\tilde{C}}=\frac{-1}{\sigma_2i_{010}}\Big[(i_{100}-\sigma_1)[-6i_{030}\sigma_2^3] \Big],\\
	&g_{4\tilde{M}\tilde{M}\tilde{C}}=\frac{-1}{\sigma_2i_{010}}\Big[(i_{100}-\sigma_1)[2i_{120}i_{010}\sigma_2^2+6i_{030}\sigma_2^2(\sigma_1-i_{100})] \Big],\\
	&g_{4\tilde{M}\tilde{C}\tilde{C}}=\frac{-1}{\sigma_2i_{010}}\Big[(i_{100}-\sigma_1)[-4i_{120}\sigma_2i_{010}(\sigma_1-i_{100})-6i_{030}\sigma_2(\sigma_1-i_{100})^2 ] \Big].
\end{align*}

The following discriminatory value $\Psi$, must not be zero for the map \eqref{paper6-eq19} to experience Neimark-Sacker bifurcation:
\begin{align*}
	\Psi=\left. \left[-\textrm{Re} \left(\frac{(1-2\lambda)\bar{\lambda}^2}{1-\lambda}\varsigma_1\varsigma_2 \right) -\frac{1}{2}|\varsigma_2|^2-|\varsigma_3|^2+\textrm{Re}(\bar{\lambda}\varsigma_4) \right] \right|_{\delta_2^*=0},
\end{align*}
where
\begin{align*}
	\varsigma_1&=\frac{1}{8}\Big[(g_{3\tilde{M}\tilde{M}}-g_{3\tilde{C}\tilde{C}}+2g_{4\tilde{M}\tilde{C}})+i(g_{4\tilde{M}\tilde{M}}-g_{4\tilde{C}\tilde{C}}-2g_{3\tilde{M}\tilde{C}} ) \Big],\\
	\varsigma_2&=\frac{1}{4}\Big[(g_{3\tilde{M}\tilde{M}}+g_{3\tilde{C}\tilde{C}})+i(g_{4\tilde{M}\tilde{M}}+g_{4\tilde{C}\tilde{C}})\Big],\\
	\varsigma_3&=\frac{1}{8}\Big[(g_{3\tilde{M}\tilde{M}}-g_{3\tilde{C}\tilde{C}}-2g_{4\tilde{M}\tilde{C}})+i(g_{4\tilde{M}\tilde{M}}-g_{4\tilde{C}\tilde{C}}+2g_{3\tilde{M}\tilde{C}})\Big],\\
	\varsigma_4&=\frac{1}{16}\Big[(g_{3\tilde{M}\tilde{M}\tilde{M}}+g_{3\tilde{M}\tilde{C}\tilde{C}}+g_{4\tilde{M}\tilde{M}\tilde{C}}+g_{4\tilde{C}\tilde{C}\tilde{C}})+i(g_{4\tilde{M}\tilde{M}\tilde{M}}+g_{4\tilde{M}\tilde{C}\tilde{C}}-g_{3\tilde{M}\tilde{M}\tilde{C}}-g_{3\tilde{C}\tilde{C}\tilde{C}} ) \Big].
\end{align*}


\begin{thebibliography}{6}
	\bibitem{paper6-Li}Li, X., Wang, H., Zhang, Z., Hastings, A.: Mathematical analysis of coral reef models. Journal of Mathematical Analysis and Applications \textbf{416}(1), 352-373 (2014)
	\bibitem{paper6-Mora}Mora, C., Graham, N. A., Nystr\"{o}m, M.: Ecological limitations to the resilience of coral reefs. Coral Reefs \textbf{35}, 1271-1280 (2016)
	\bibitem{paper6-Doering}Doering, T., Wall, M., Putchim, L., Rattanawongwan, T., Schroeder, R., Hentschel, U., Roik, A.: Towards enhancing coral heat tolerance: a “microbiome transplantation” treatment using inoculations of homogenized coral tissues. Microbiome \textbf{9}(1), 102 (2021)
	\bibitem{paper6-Ranjit}Ranjit, B., Biswas, S., Bhattacharyya, J., Chattopadhyay, J.: Dynamics of Zooplankton-Mediated Disease Outbreak in Coral-reef. Differential Equations and Dynamical Systems 1-29 (2023)
	\bibitem{paper6-Clements}Clements, C. S., Hay, M. E.: Disentangling the impacts of macroalgae on corals via effects on their microbiomes. Frontiers in Ecology and Evolution \textbf{11}, 1083341 (2023)
	\bibitem{paper6-Jokiel}Jokiel, P. L., Morrissey, J. I.: Influence of size on primary production in the reef coral Pocillopora damicornis and the macroalga Acanthophora spicifera. Marine Biology  \textbf{91}, 15-26 (1986)
	\bibitem{paper6-Melis}Melis, R., Ceccherelli, G., Piazzi, L., Rustici, M.: Macroalgal forests and sea urchin barrens: Structural complexity loss, fisheries exploitation and catastrophic regime shifts. Ecological Complexity \textbf{37}, 32-37 (2019)
	\bibitem{paper6-Panja}Misra, A. K., Tiwari, P. K., Chandra, P.: Modeling the control of algal bloom in a lake by applying some external efforts with time delay. Differential Equations and Dynamical Systems \textbf{29}, 539-568 (2021)
	\bibitem{paper6-May}May, R. M.: Biological populations with nonoverlapping generations: stable points, stable cycles, and chaos. Science \textbf{186}(4164), 645-647 (1974)
	\bibitem{paper6-Zhang}Zhang, F., Tian, H., Zhao, H., Zhang, X., Shi, Q.: Spatiotemporal pattern formation in a discrete toxic-phytoplankton-zooplankton model with cross-diffusion and weak Allee effect. International Journal of Bifurcation and Chaos \textbf{32}(10), 2250156 (2022)
	\bibitem{paper6-Han}Han, X., Lei, C.: Bifurcation and turing instability analysis for a space-and timediscrete predator-prey system with Smith growth function. Chaos, Solitons \& Fractals \textbf{166}, 112910 (2023)
	\bibitem{paper6-Xu}Xu, W., Liu, X., Wang, H., Zhou, Y.: Event-based optimal output-feedback control of nonlinear discrete-time systems. Information Sciences \textbf{540}, 414-434 (2020)
	\bibitem{paper6-Alamin}Atabaigi, A.: Multiple Bifurcations and Dynamics of a Discrete Time Predator-Prey System with Group Defense and Non-Monotonic Functional Response. Differential Equations and Dynamical Systems \textbf{28}(1), 107-132 (2020) 
	\bibitem{paper6-Zhu}Ren, J., Yu, L.: Codimension-two bifurcation, chaos and control in a discrete-time information diffusion model. Journal of Nonlinear Science \textbf{26}, 1895-1931 (2016)
	\bibitem{paper6-Zelinka}Zelinka, I., Diep, Q.B., Snášel, V., Das, S., Innocenti, G., Tesi, A., Schoen, F. and Kuznetsov, N.V.: Impact of chaotic dynamics on the performance of metaheuristic optimization algorithms: An experimental analysis. Information Sciences \textbf{587}, 692-719 (2022)
	\bibitem{paper6-Murray}Murray, J. D.: Mathematical biology: I. An introduction. Interdisciplinary applied mathematics. Mathematical Biology, Springer \textbf{17}, (2002)
	\bibitem{paper6-Luo1}Luo, X. S., Chen, G., Wang, B. H., Fang, J. Q.: Hybrid control of period-doubling bifurcation and chaos in discrete nonlinear dynamical systems. Chaos, Solitons \& Fractals \textbf{18}(4), 775-783 (2003)
	\bibitem{paper6-Kuznetsov}Kuznetsov, Y. A., Kuznetsov, I. A., Kuznetsov, Y.: Elements of applied bifurcation theory. New York: Springer \textbf{112}, (1998)
	\bibitem{paper6-Mukherjee}Nath, B., Kumari, N., Kumar, V., Das, K. P.: Refugia and Allee effect in prey species stabilize chaos in a tri-trophic food chain model. Differential Equations and Dynamical Systems 1-27 (2019)
	\bibitem{paper6-Rahman}Kumari, N., Kumar, V.: Controlling chaos and pattern formation study in a tritrophic food chain model with cannibalistic intermediate predator. The European Physical Journal Plus \textbf{137}(3), 1-23 (2022)
	\bibitem{paper6-Singh}Singh, A., Deolia, P.: Dynamical analysis and chaos control in discrete-time prey-predator model. Communications in Nonlinear Science and Numerical Simulation \textbf{90}, 105313 (2020)
	\bibitem{paper6-Akhtar}Akhtar, S., Ahmed, R., Batool, M., Shah, N. A., Chung, J. D.: Stability, bifurcation and chaos control of a discretized Leslie prey-predator model. Chaos, Solitons \& Fractals \textbf{152}, 111345 (2021)
	\bibitem{paper6-Robinson}Robinson, C.: Dynamical systems: stability, symbolic dynamics, and chaos. CRC press (1998)
	\bibitem{paper6-Layek}Layek, G. C.: An introduction to dynamical systems and chaos. New Delhi: Springer (2015)
	\bibitem{paper6-Carr}Carr, J.: Applications of centre manifold theory. Springer Science \& Business Media (2012)
	\bibitem{paper6-Guckenheimer}Guckenheimer, J. Holmes, P.: Nonlinear oscillations, dynamical systems, and bifurcations of vector fields. Springer Science \& Business Media (2013)
	\bibitem{paper6-Fattahpour}Fattahpour, H., Zangeneh, H. R., Wang, H.: Dynamics of coral reef models in the presence of parrotfish. Natural Resource Modeling \textbf{32}(2), e12202 (2019)
	\bibitem{paper6-Blackwood}Blackwood, J. C., Hastings, A.: The effect of time delays on Caribbean coral–algal interactions. Journal of Theoretical Biology \textbf{273}(1), 37-43 (2011)
	\bibitem{paper6-Blackwood1}Blackwood, J. C., Hastings, A, Mumby, P. J.: The effect of fishing on hysteresis in Caribbean coral reefs. Theoretical ecology \textbf{5}, 105-114 (2012)
	\bibitem{paper6-Rivero}Gonz\'{a}lez-Rivero, M., Yakob, L., Mumby, P. J.: The role of sponge competition on coral reef alternative steady states. Ecological Modelling \textbf{222}(11), 1847-1853 (2011).
	\bibitem{paper6-LuoAC}Luo, A. C.: Regularity and complexity in dynamical systems. New York: Springer (2012)
	\bibitem{paper6-Salman}Salman, S. M., Yousef, A. M., Elsadany, A. A.: Stability, bifurcation analysis and chaos control of a discrete predator-prey system with square root functional response. Chaos, Solitons \& Fractals \textbf{93}, 20-31 (2016)
	\bibitem{paper6-He}He, Z., Lai, X.: Bifurcation and chaotic behavior of a discrete-time predator–prey system. Nonlinear Analysis: Real World Applications \textbf{12}(1), 403-417 (2011)
	\bibitem{paper6-Elabbasy}Elabbasy, E. M., Elsadany, A. A., Zhang, Y.: Bifurcation analysis and chaos in a discrete reduced Lorenz system. Applied Mathematics and Computation \textbf{228}, 184-194 (2014)
	\bibitem{paper6-Liu}Liu, W., Jiang, Y.: Flip bifurcation and Neimark–Sacker bifurcation in a discrete predator–prey model with harvesting. International Journal of Biomathematics \textbf{13}(01), 1950093 (2020)
	\bibitem{paper6-Abdelaziz}Abdelaziz, M. A., Ismail, A, I., Abdullah, F. A., Mohd, M. H.: Bifurcations and chaos in a discrete SI epidemic model with fractional order. Advances in Difference Equations \textbf{2018}(1), 1-19 (2018)
	\bibitem{paper6-Ott}Ott, E., Grebogi, C., Yorke, J. A.: Controlling chaos. Physical review letters \textbf{64}(11), 1196 (1990)
	\bibitem{paper6-Luo}Luo, A. C.: Bifurcation and stability in nonlinear dynamical systems. Springer International Publishing (2019) 
	\bibitem{paper6-Yuan}Yuan, L. G., Yang, Q. B.: Bifurcation, invariant curve and hybrid control in a discrete-time predator–prey system. Applied Mathematical Modelling \textbf{39}(8), 2345-2362 (2015)
\end{thebibliography}
\end{document}